\documentclass[a4paper,11pt]{article}
\usepackage{amsmath}
\usepackage{amsfonts}
\usepackage{amsthm}
\usepackage{hyperref}
\usepackage{array}
\usepackage[hang,small]{caption}
\usepackage{graphicx}
\usepackage{tikz}
\usepackage{natbib}
\setcitestyle{authoryear,open={},close={}}

\usepackage{mathrsfs}
\usepackage [english,french] {babel}
\usepackage[utf8]{inputenc}

\newcommand{\R}{\mathbb{R}}

\newtheorem{de}{Definition}[subsection] 
\newtheorem{theo}{Theorem}[section]

\newtheorem{properties}[theo]{Properties}

\newtheorem{remark}{Remark}
\newtheorem{lemma}{Lemma}[section]

\title{Bistability induced by generalist natural enemies can reverse pest invasions 
}

\author{\textsc{Sten Madec}$^{a}$, \textsc{J\'{e}r\^{o}me Casas}$^{b,c}$, \textsc{Guy Barles}$^{a}$, \textsc{Christelle Suppo}$^{b}$\\
[1mm] $^{a}${\small \it  LMPT, UMR CNRS 7350, Universit\'e François-Rabelais de Tours, 37200 Tours - France}\\
$^{b}${\small \it  IRBI, UMR CNRS 7261, 
Université François-Rabelais de Tours,  37200 Tours, France}\\
$^{c}${\small \it Institut Universitaire de France}
}

\begin{document}
\maketitle



{\small {\noindent \bf Abstract.}
\noindent
Reaction-diffusion analytical modeling of predator-prey systems has shown that specialist natural enemies can slow, 
stop and even reverse pest invasions, assuming that the prey population displays a strong Allee effect in its growth. 
Few additional analytical results have been obtained for other spatially distributed predator-prey systems, as traveling waves of non-monotonous systems  are notoriously difficult to obtain.
Traveling waves have indeed recently been shown to exist in predator-prey systems, but the direction of the wave, 
an essential item of information in the context of the control of biological invasions, is generally unknown. 
Preliminary numerical explorations have hinted that control by generalist predators might be possible 
for prey populations displaying logistic growth. We aimed to formalize 
the conditions in which spatial biological control can be achieved by generalists, 
through an analytical approach based on reaction-diffusion equations.\\


The population of the focal prey — the invader — is assumed to grow according to a logistic function. 
The predator has a type II functional response and is present everywhere in the domain, at its carrying capacity, 
on alternative hosts. Control, defined as the invader becoming extinct in the domain, may result from spatially 
independent demographic dynamics or from a spatial extinction wave.
Using comparison principles, we obtain sufficient conditions for control and for 
invasion, based on scalar bistable partial differential equations (PDEs).
The searching efficiency and 
functional response plateau of the predator are identified as the main parameters defining the parameter space for prey extinction and invasion. 
Numerical explorations are carried out in the region of those control parameters space between the super- and subsolutions, 
in which no conclusion about controllability can be drawn on the basis of analytical solutions. \\

 The ability of generalist predators to control prey populations with logistic growth lies in the bistable dynamics of the coupled system, 
rather than in the bistability of prey-only dynamics as observed for specialist predators attacking prey populations displaying Allee effects.  
Analysis of the ordinary differential equations (ODEs) system identifies parameter regions with monostable (extinction) and 
bistable (extinction or invasion) dynamics. By contrast, analysis of the associated PDE system distinguishes different and additional 
regions of invasion and extinction. Depending on the relative positions of these different zones, four patterns of spatial dynamics can be identified:
traveling waves of extinction and invasion, pulse waves of extinction and heterogeneous stationary positive solutions of the Turing type. 
As a consequence, prey control is predicted to be possible when space is considered in additional 
situations other than those identified without considering space. The reverse situation is also possible. None of these considerations apply to spatial
predator-prey systems with specialist natural enemies. 
The consideration of space in predator-prey systems involving generalist predators with a parabolic functional response
is thus crucial.
}

\section{Introduction}

\subsection{ Modeling the biological control of invasive pests }

 Biological invasions are a major contemporary problem (\cite{Pimentel,Garnier,Mistro,Potapov,Wang,Savage})
 for which few solutions are available, all of which are very costly.
The use of natural enemies for the biological control of invading insects is one of the most promising possibilities 
(\cite{Moffat,Li,Ye,Basnet}).
As invasion is essentially a spatial process, the potential of natural enemies to stop or even reverse an invasion is of particular interest.
The fundamental analytical work of \cite{owenlewis} showed that {\it specialist} predators could potentially slow, stop or reverse the spread of
invasive pests.
The reversal of pest spread by specialist predation requires a strong Allee effect for the pest-only dynamics, 
defined as a negative growth rate for the prey population at low density. In the presence of a weak Allee effect, 
the predator can stop, but not reverse the wave of invasion. These conclusions have been confirmed in several other theoretical studies (\cite{Cai2014,Boukal2007,Petro2009}).

Generalist predators can also control prey effectively (\cite{Ebarch2014,Chakraborty}). 
Their use could be promoted through conservation biological control programs without the need for exogenous specialist natural enemies. 
Unfortunately, the role of generalist predators in the spatial control of their prey has been much less studied than that of specialist predators,
due to the intrinsic difficulties of having to work with a {\it system} of equations rather than with a {\it single} scalar equation. However, 
two important studies have been carried out in this area: the analytical and comprehensive study of \cite{DuShi2007}, 
and the preliminary simulation study of \cite{Fagan2002}. Both used the same model structure as we do here, with logistic growth for both prey and predator populations,
and a type II functional response for predators. The convergence of these models was strengthened further by the in-depth analysis of \cite{Magal}
in which space was not considered. It is difficult to use these models in a spatial context: the work of \cite{DuShi2007} cannot deal with invasion
and traveling waves, because it deals with a bounded space. The numerical simulations of \cite{Fagan2002} are restricted to a few parameter values. 
They are, however, valuable, because they suggest conditions in which a generalist predator might be able to stop, and even reverse 
the invasion wave of a pest population displaying logistic growth. {Fagan {\it et al}. also reported the results of field studies indicating that predators
with diffusion coefficients higher than those of their prey are poor control organisms. The authors provided an explanation for this finding founded on 
logical arguments, but without a firm mathematical foundation.} This result has been confirmed by a few numerical simulations including space, as reported by
\cite{Magal}, revealing a strong dependence of system dynamics on the relative rates of diffusion of the prey and the predator. It is thus important
to take space into account, by considering diffusion coefficients of both predator and prey. This conclusion accounts for the interest of scientists in questions of this type (\cite{Lewis2013,Hastings2000,Roos1991,Roos1998}).

There are therefore hopes that it might be possible to extend the conditions for the control of invasive prey organisms to (i)
generalist predators and (ii) prey populations displaying growth patterns not dependent on the restrictive assumption of Allee effects.
Such control approaches would have a major impact in the field, {given the high degree of generalism obtained}. The aim of this study was, therefore to formalize the conditions in 
which spatial biological control can be achieved by generalists, through an analytical approach based on traveling waves solutions
of reaction-diffusion equations. \\
%
%
Traveling wave solution  describes a constant profile $U$ moving through space at a speed $c$.
Such waves are  often observed in nonlinear reaction-diffusion systems modeling various phenomena. 
They  are particularly suitable for describing the propagation of invasive fronts.
In systems modeling a single species, described by  a {\it scalar} equation, this type of solution is  very well understood 
(\cite{Fischer,KPP} and \cite{volperts} for {a} complete theory). 
Two particular classes of equations can be distinguished: monostable equations (like the Fisher-KPP equations)
and bistable equations (often modeling the Allee effect).
In monostable equations, there is a minimal wave speed $c^*$ such that, for any $c\geq c^*$, a wave solution  
with speed $c$ exists. 
In bistable equations traveling waves exist for {\it a unique} speed $c^*$. 
{The sign of this speed $c^*$ distinguishes between invasion or extinction of prey, 
which  is a key property for our purposes.}

For interactions of several species (described by a {\it multidimensional} system), the situation is much more complex. 
However, for some type of interaction, cooperation for instance, the system possesses a strong structural property, namely monotonicity.
Essentially, this monotony makes it possible to use the comparison principle, which is always possible for one-dimensional systems, and the theory 
is then complete  (see \cite{volperts}).
Unfortunately, our system, and prey-predator systems in general, do not have {such a} monotonous structure.
This method is then unsuitable for monotonous systems and only a few results have been published. 
 One of the key reasons for this is as follows: when
we search for traveling wave solutions for a system with $N$ equations, we obtain a system of $N$ second order ordinary 
equations that
can be reformulated as a system of $2N$ first order ordinary equations. In the scalar case ($N=1$)
, it is therefore possible to study trajectories in a plane,{ available}
using classical tools for two-dimensional dynamical systems. 
For several species ($N>1$),  it is necessary to study trajectories in a $2N$-dimensional space, which may be very difficult.

 Hence, the first rigorous results demonstrating the existence of traveling waves in prey-predator systems were based on a generalization,
 to the fourth dimension, of the classical 
 shooting method in the phase plane
(\cite{Dunbar1,Dunbar2}). This approach has since been generalised  (\cite{Huang,Xuweng}). 
However, all these studies simply investigate the mere {\it existence} of traveling waves. 
They do not determine the direction of the wave or the global dynamics for general initial conditions.
Other methods have recently been developed  in similar models 
(\cite{HuangWeng,Ducrot1}),  but they are subject to the same limitations. 
A last  approach is to  use the  degree theory  (see e.g. \cite{Gio,volperts}) to obtain the existence of traveling waves. These homotopy methods  may  occasionally give some information on the speed $c$. Unfortunately, this needs additional estimates which are very difficult to obtain here. We therefore required another method.
%
%
%
%

The analysis provided in  \cite{Magal} gives conditions for 
preys' control by predators, but this analysis was carried out largely without reference to space.
Thus, we have extended the system of  \cite{Magal} by adding spatial diffusion.
We find that the conditions for control are very different from those for the system in which space is not considered. 
The conditions for prey extinction and invasion are discussed in terms of two essential parameters:
the encouter rate $E$ and the handling time $h$. 
Increasing $E$ clearly increases predator pressure. Conversely, increasing $h$ decreases predator pressure.\\

The paper is organized as follows. 
In section  \ref{sec2} we present the mathematical model and the main result of this work: theorem \ref{thODE} 
describes invasion conditions for the ODE system and the theorems \ref{thcontrol} and \ref{thnocontrol} 
the invasion conditions for the PDE system. 
The mathematical results are completed by numerical simulations in section \ref{simu}.
The results are discussed in section \ref{ccl}.
The final section \ref{Proofs} is devoted to the mathematical proofs.

\section{Model and main results}\label{sec2}

\subsection{Mathematical model}\label{secadim}

We analyze a system of partial differential equations for a prey population with logistic growth, and a generalist predator
population with logistic growth on alternative prey in the absence of the invading host. 
The functional response is of Holling type II. 
The prey-predator interactions are modeled by the following partial differential equation system: 
\begin{equation}
\left\{\begin{array}{lr}
\partial_t u=D_u \Delta_x u+r_1 u\left(1-\frac{u}{K_1}\right)-\frac{Euv}{1+Ehu},&\\
\partial_t v=D_v \Delta_x v+r_2 v\left(1-\frac{v}{K_2}\right)+\gamma\frac{Euv}{1+Ehu}, & x \in \mathbb{R},\; t \in \mathbb{R}^{+} \\
u(x,0)=u_0(x),\quad v(x,0)=1 &  
\end{array}\right.
\end{equation}
with:\\
\begin{tabular}{ll}
$ u (t,x) =$ prey density at time t and at point x. &
 $ v (t,x) =$ predator density at time t and at point x.\\
 $D_u =$ diffusion rate of prey&
$ D_v =$ diffusion rate of predators\\
$ r_1 =$ growth rate of prey&
$ r_2 =$ growth rate of predators\\
$K_1 =$ carrying capacity of prey&
$ K_2 =$ carrying capacity of predators in absence of focal prey\\
$ E =$ encounter rate&
$ h =$ handling time\\
$\gamma =$ conversion efficiency& $u_0\geq 0$ the initial concentration of prey, \\
\end{tabular}
$D_u, D_v, r_1, r_2, K_1, K_2, E, h \text{ and } \gamma$ are positive constant parameters.\\

\noindent We carried out the following adimensionalization: 
\begin{itemize}
\item[] $t'=r_1 t$; $x'=x\sqrt{\frac{r_1}{D_u}}$; $u'(x',t')=u(t,x)/K_1$ ; $v'(x',t')=v(t,x)/K_2$
\item[]$ d'=D_v/D_u$ ; $r'=r_2/r_1$ ; $E'= E K_2/r_1$ ; $h'=r_1 h K_1/K_2$ ; $\gamma'=\gamma K_1/K_2$ ; $\alpha=\frac{\gamma'}{r'}$.
\end{itemize}
Removing the sign ’ to simplify the notation, the system reads

\begin{equation}\label{sys2}
\left\{\begin{array}{lr}
\partial_t u=  \Delta_x u+ u\left(1-u\right)-\frac{Euv}{1+Ehu},& x \in \mathbb{R}, t \in \mathbb{R}^{+}\\
\partial_t v=d \Delta_x v+r \left( v\left(1-v\right)+\alpha\frac{Euv}{1+Ehu}\right)&\\
\end{array}\right.
\end{equation}
 with the initial conditions\footnote{ 
 All our results remain true for various different initial conditions. 
 The essential condition is  that the solutions of the scalar systems we consider converge to traveling wave solutions.
In particular, compact support may be allowed for  $u_0$. See \cite{fife} for a detailed discussion.
}
\begin{equation}\label{CI}
\left\{ 
\begin{array}{l}

 u(0,x)=u_0(x)\in[0,1]\;;\; \lim\limits_{x\to-\infty} u_0(x)=1\;;\; \lim\limits_{x\to+\infty} u_0(x)=0\\
v(0,x)=1. 
\end{array}\right.
\end{equation}

\subsection{Main results}\label{mains}
%
We distinguish two ways in which a predator can control the prey, one taking space into account and the other not considering this factor
(mathematical definitions are provided in definition \ref{definv}).

\begin{itemize}
\item The spatially {\it uniform extinction} results exclusively from {\it local} demographic processes and is independent of space.

\item The {\it extinction wave} is due to both  demographic  and  diffusive processes 
and may take various forms, from a traveling front to a pulse. 
\end{itemize}
\noindent Conversly, invasion is defined as prey survival and we distinguish two ways in which the prey can invade. 
\begin{itemize}
\item The spatially {\it uniform invasion}, which is independent of  space.
\item
The {\it non-uniform  invasion},  described by various {\it spatial} dynamics, 
from Turing phenomena to invasion waves.

\end{itemize}

\begin{de}\label{definv}
 Let  $(u_0(x),v_0(x))$ be an initial condition verifying \eqref{CI} and
  $(u(t,x),v(t,x))$ be the corresponding solution of \eqref{sys2}.
 \begin{itemize}
  \item  {\it Extinction} of  prey  occurs if  
 $$\forall\, x\in\R,\lim_{t\to+\infty} u(t,x)=0.$$
 \item Prey extinction  is   {\it uniform}  if it is uniform with respect to $x\in\R$, that is, if  there exists a map $\phi(t)$ verifying
 $$\forall\, x\in\R,\,\forall t>t_0,\; 0\leq u(t,x)\leq \phi(t)\text{ and } \lim_{t\to+\infty} \phi(t)=0.$$
 \item Prey extinction  is {\it non uniform}  if 
there is extinction but no  uniform extinction.
\item {\it Invasion} of prey  occurs if there is no extinction, that is if 
$$\exists x\in\R,\; \limsup_{t\to+\infty}  u(t,x)>0$$
 \end{itemize}
\end{de}

\subsubsection{Analysis of the associated ODE system}\label{mainsode}

\noindent If space is not taken into account,  system \eqref{sys2} may be rewritten as follows.

\begin{equation}\label{sysODE}
\left\{\begin{array}{lr}
 \frac{d}{dt}u=  u\left(1-u\right)-\frac{Euv}{1+Ehu},&  t \in \mathbb{R}^{+}\\
\frac{d}{dt} v=r \left(v\left(1-v\right)+\alpha\frac{Euv}{1+Ehu}\right).&\\
0 < u(0)=u_0  \leq 1 \leq v(0)=v_0   & 
\end{array}\right.
\end{equation}
System ~\eqref{sysODE} is well understood (\cite{Magal}). 
Indeed, it is clear that there are always three trivial stationary states: $(0,0)$ and $(1,0)$, which are unstable 
and $(0,1)$, which is asymptotically stable if, and only if, $E>1$.
Moreover, there are no more than three non-trivial positive steady states. We are interested principally in the case $E>1$.
In this case, there are either no or two stationary positive steady states. If the two steady states exist, denoted $(\widehat{u},\widehat{v})$ and $(u^*,v^*)$
with  $\widehat{u}<u^*$ and $\widehat{v}<v^*$, then $(\widehat{u},\widehat{v})$ is always unstable 
and $(u^*,v^*)$ is most often  stable. In this case, there are two stable nonnegative solutions, $(0,1)$ and $(u^*,v^*)$ and the system is bistable.

We are interested principally in the conditions for prey extinction.
If $E<1$, then $(0,1)$ is unstable and no extinction occurs.  We are therefore interested only in the case $E>1$.
Now, if $E>1$, there are two possibilities. 
In the non bistable case, $(0,1)$ is globally stable and  there is extinction.
In the bistable case, provided that $u$ is initially small enough, say $u_0<\mu_1$ for some $0<\mu_1<1$, 
then $u(t)\rightarrow 0$ as $t\rightarrow +\infty$. 
Conversely, if $u$ is initially large enough, say $0<\mu_2<u_0$, then $ u(t) \to u^*$ as $t\rightarrow +\infty$. 
Thus, in this case, the outcome — extinction or invasion — depends on the initial conditions.
 The following result provides an explicit statement of the above in the parameter space $(E\;;\; h)$ and is proven in section \ref{proofode}.
\begin{theo}\label{thODE}
Let $E>1$ and $\alpha\geq 0$ be fixed. 
\begin{itemize}
 \item[(i)] {\bf Existence of positive solutions.}
There exists a unique $h^*=h^*(E,\alpha)$ such that 
\begin{itemize}
 \item 
If $h< h^*$,  then there is no positive stationary solution and  there is  extinction of prey for the ODE system.
\item If $h>h^*(E,\alpha)$, then there exist two positive solutions for the ODE system $(\widehat{u},\widehat{v})$ and $(u^*,v^*)$ with $\widehat{u}<u^*$. 
\end{itemize}
\item[(ii)] {\bf Stability of the solutions.}
Let $h>h^*(E,\alpha)$. The solution $(\widehat{u},\widehat{v})$ is always unstable. \\
Moreover, there exists a unique $h^{**}(E,\alpha)>h^*(E,\alpha)$ such that
\begin{itemize}
\item If $h>h^{**}(E,\alpha)$ then $(u^*,v^*)$ is stable.
\item If $h^*(E,\alpha)<h<h^{**}(E,\alpha)$, the stability of $(u^*,v^*)$ depends on $r$. It is unstable if $r$ is small enough and stable otherwise.
\end{itemize}
\end{itemize}
\end{theo}
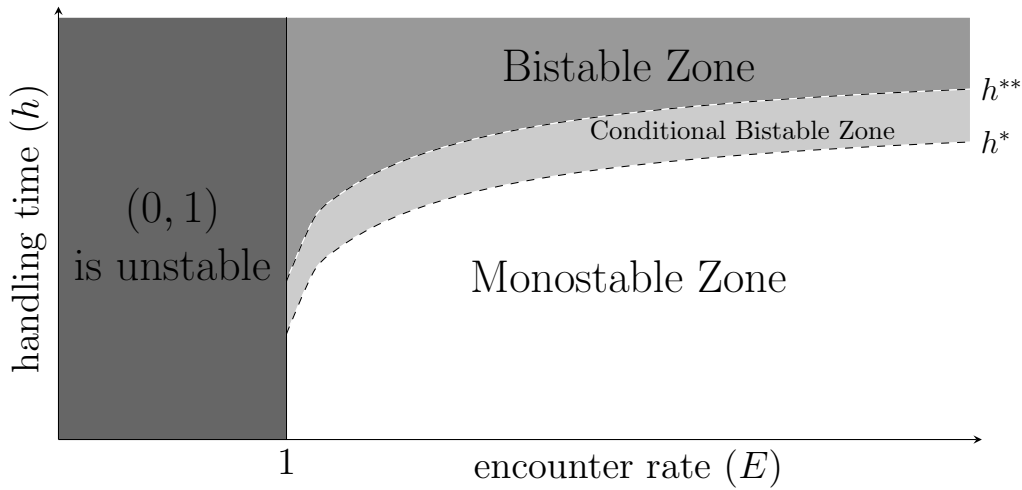
\begin{figure}[h]
\centering
\begin{tikzpicture}[x=3cm,y=0.7cm,>=stealth]
\filldraw[draw=white,fill=black!40] 
 plot[domain=1:4] (\x,8)
 -- plot [domain=4:1,smooth] (\x, {(2-(\x)^(-1)+(10/3)*sqrt(1-(\x)^(-1)))+2})
 -- cycle ;
\filldraw[draw=white,fill=black!60] 
     plot [domain=0:8] (0.0, \x)
  -- plot [domain=8:0] (1, \x)
  -- cycle;
\filldraw[draw=white,fill=black!20] 
     plot  [domain=1:4,smooth] (\x, {(2-(\x)^(-1)+(10/3)*sqrt(1-(\x)^(-1)))+1})
  -- plot   [domain=4:1,smooth] (\x, {(2-(\x)^(-1)+(10/3)*sqrt(1-(\x)^(-1)))+2})
  -- cycle;
\draw [->] (0.0,0) -- (4.05,0);
\draw [->] (0.0,0) -- (0.0,8.2);
\draw [-] (1,0) -- (1,8);
\draw [dashed] plot [domain=1:4,smooth] (\x, {(2-(\x)^(-1)+(10/3)*sqrt(1-(\x)^(-1)))+1})node [right, sloped] { \large $h^*$};
\draw [dashed] plot [domain=1:4,smooth] (\x, {(2-(\x)^(-1)+(10/3)*sqrt(1-(\x)^(-1)))+2})node [right, sloped] { \large $h^{**}$};
\draw (2.5,0) node [below] {\Large encounter rate ($E$)};
\draw (0,4) node [above,rotate=90] {\Large handling time ($h$)};
\draw (1,0) node [below] {\Large 1};
\draw (0.5,2.5) node [above]{ {\LARGE \begin{tabular}{c}$(0,1)$\\ is unstable\end{tabular}}};
 \draw (2.5,2.6) node [above,rotate=0]{ \LARGE Monostable Zone};
 \draw (2.5,6.6) node [above,rotate=0]{ \LARGE Bistable Zone};
  \draw (3,5.5) node [above,rotate=0]{\small Conditional Bistable Zone};
\end{tikzpicture}
\caption{\label{fig0}{
Description of the dynamics of the ODE system \eqref{sysODE} in the $E-h$ plane. 
If $E<1$, the control solution $(0,1)$ is unstable. In this zone there exists at least one positive steady state and the prey never disappears entirely. 
If $E>1$, then the control solution $(0,1)$ is always (locally) stable. 
Moreover, the  $E>1$ zone is the union of three subzones. 
Below the $h^*$ curve, $(0,1)$ is the only non-negative steady state and is a global attractor: this is a monostable zone. 
Above the $h^*$ curve, there are two additional 
positive steady states, one of which is always unstable while the second, denoted $(u^*,v^*)$, may be stable or unstable.
Above the $h^{**}$ curve, $(u^*,v^*)$ is  always  stable.
In this subzone,  the asymptotic behavior depends only on the initial conditions: this is a bistable zone. 
Between the  $h^*$ and the $h^{**}$ curves, the stability of $(u^*,v^*)$ depends on other parameters: 
this is a conditional bistable zone. 
For illustrative purpose, the size of this last subzone has been considerably increased.
}}
\end{figure}
\begin{remark}
 Our calculations show that the gap between $h^*$ and $h^{**}$ is very small, so that, roughly speaking, $(u^*,v^*)$ is stable whenever it exists. 
 However, if $h$ belongs to the conditional stability zone, i.e. $h\in(h^*,h^{**})$, and  if $r$ is very small, 
stability is lost and the system becomes excitable.
This explains, in particular, the presence of pulses for small values of $r$ when dealing with spatial interactions (see section \ref{simu}). 
\end{remark}

\noindent The map $(E,\alpha)\mapsto h^*(E,\alpha)$ has the following properties, as proved in section \ref{proofode}.
\begin{properties}\label{prophet}
Let $\alpha\geq 0$ be fixed. The map $E\mapsto h^*(E,\alpha)$ is increasing and one has the explicit limits:
$$\lim_{E\to 1} h^*(E,\alpha)=\left\{\begin{array}{c}1+\alpha\;\text{ if $\alpha<1$}\\2\sqrt{\alpha}\;\text{ if $\alpha\geq 1$} \end{array}\right.\quad;\quad 
\lim_{E\to+\infty} h^*(E,\alpha)=2+2\sqrt{1+\alpha}.$$
Let $E\geq1$ be fixed. The map $\alpha\mapsto h^*(E,\alpha)$ is increasing and one has the explicit limits:
$$h^*(E,0)=\frac{1}{E}\left(2E-1+2\sqrt{E(E-1)}\right):=h_1(E)\quad\text{and}\quad \lim_{\alpha\to+\infty} h^*(E,\alpha)=+\infty.$$
\end{properties}

\noindent Figure \ref{fig0} illustrates the maps $h^*$ and $h^{**}$ and the possible outcomes for  system \eqref{sysODE}.

\subsubsection{Analysis of the PDE system}\label{mainspde}

We wish to identify the parameter conditions required to obtain prey extinction in the PDE system \eqref{sys2}.
A simple stability analysis shows that, if $E<1$, then invasion occurs in the PDE~system. If $E>1$, then the situation for the PDE~system is more complex. 
In this situation, the spatial structure and diffusion processes result in additional conditions for extinction. 
The rationale is explained {in detail below}.

Let us assume that there is a positive stable stationary solution of \eqref{sys2} denoted by $(u^*,v^*)$ and that the initial condition $u(x,0)$ 
is close to $u^*$ at some places $x$ and close to $0$ at  other places. 
Since both $u^*$ and $0$ are stable, the demographic phenomena lead to an agregation near $u^*$ and an agregation near $0$. 
However, diffusion allows individuals to move around in space, so one of $0$ or $u^*$ may be the final global attractor. 
In other words, there may be a (stable) traveling wave joining $u^*$ to $0$.
The direction of this wave, given by the sign of the speed of the wave, indicates whether extinction or invasion occurs. 
However, there are difficulties associated with this argument.

\begin{itemize}
\item There can be no homogeneous stationary solution of \eqref{sys2}, only stable heterogeneous positive stationary solutions.
In other words, it is possible that $h<h^*(E,\alpha)$ without control occurring. 
\item Even in the case of bistability ($h>h^*(E,\alpha)$), the bistable system \eqref{sys2} is neither competitive nor cooperative. 
Little theoretical knowledge 
is available concerning the occurrence of traveling waves in such systems, with even less known about the stability and direction of the wave. 

\end{itemize}
Using super and subsolutions, we show here how to obtain the conditions sufficient (but not necessary) for extinction and for invasion, 
based on well known {\it scalar} bistable PDEs.
Roughly speaking,  let $(u(t,x),v(t,x))$ be the solution of \eqref{sys2}. If we find a {\it positive constant}  $\underline{v}$ 
such that, for any
\footnote{It suffices that this condition occurs for $t>t_0$ for some $t_0>0$.} $(t,x)\in \R^+\times \R$, $v(t,x)\geq \underline{v}$ then it comes
$$\partial_t u(t,x)-\Delta_x u(t,x)\leq u(t,x)(1-u(t,x))-\frac{E u(t,x)}{1+Eh u(t,x)} \underline{v},\quad t>0,\; x\in\R.$$
Let $\overline{u}$ be the solution of

\begin{equation}\label{eq:sursol}
 \left\{\begin{array}{l}
\partial_t \overline{u}(t,x)-\Delta_x \overline{u}(t,x)= \overline{u}(t,x)(1-\overline{u}(t,x))-\frac{E \overline{u}(t,x)}{1+Eh \overline{u}(t,x)} 
\underline{v},\\
\overline{u}(0,x)\geq u(0,x).
\end{array}\right.
\end{equation}
The comparison principle implies that $\overline{u}(t,x)\geq u(t,x)$.
Now,  if $\overline{u}(x,t) \rightarrow 0$ when $t \rightarrow +\infty$ then $u(t,x)\rightarrow 0$ when $t \rightarrow +\infty$ 
and extinction of prey occurs (see figure \ref{fig:fig1}-(a)).
Moreover, if $\overline{u}(x,t)=\phi(t)$ does not depend on $x$, then there is aspatial control (see figure \ref{fig:fig1}-(b)).\\

\noindent Conversely, if we can identify a  {\it positive constant} $\overline{v}$  such that $v(x,t)\leq \overline{v}$, then  it comes
$$\partial_t u(t,x)-\Delta_x u(t,x)\geq u(t,x)(1-u(t,x))-\frac{E u(t,x)}{1+Eh u(t,x)} \overline{v},\quad t>0,\; x\in\R.$$
Now, define $\underline{u}(t,x)$ as the solution of 
\begin{equation}\label{eq:soussol}
 \left\{\begin{array}{l} 

\partial_t \underline{u}(t,x)-\Delta_x \underline{u}(t,x)= \underline{u}(t,x)(1-\underline{u}(t,x))-\frac{E \underline{u}(t,x)}{1+Eh \underline{u}(t,x)} 
\overline{v},\\
\underline{u}(0,x)\leq u(0,x).
\end{array}\right.
\end{equation}
The comparison principle implies that  $\underline{u}(t,x) \leq u(t,x)$. 
It follows that if for some $x\in\R$, ${\displaystyle \limsup_{t\to+\infty} \underline{u}(x,t) >0}$,  
then ${\displaystyle\limsup_{t\to+\infty}  u(t,x)>0}$ and  there is (non uniform) invasion (see figure \ref{fig:fig2}).\\

These arguments give rise to the following theorems yielding sufficient conditions, in terms of the parameters $E$, $\alpha$ and $h$,
for extinction or invasion to occur. All theorems are proven in section \ref{Proofs}.
We begin with a sufficient condition for uniform extinction. 
\begin{figure}[htb]
\centering
\begin{tikzpicture}[x=0.12cm,y=0.12cm,>=stealth]
\draw [->] (-10,0) -- (50,0);
\draw [->] (-10,0) -- (-10,35);
\draw [-, line width=1.5 pt] (-10,15) -- (0,15) node [midway, above, sloped] {\scriptsize  $u(x,t)$}.. controls (4,15) and (6,30) .. (9,25)  ..controls (10,10) and (20,0.1) .. (30,0.3) ..controls (40,0.1) .. (50,0.01)  ;
\draw [-, dashed] (-10,28) -- (25,28) node [midway,above,sloped]  {\scriptsize  $\overline{u}(x,t)$} ..controls (32,30) and (28,0) .. (35,1)  ..controls (40,0) .. (50,0);
\draw [->,line width=0.5] (33,5) -- (29,5);
\draw [->,line width=0.5] (30,25) -- (26,25);
\draw (20,0) node [below]  {space ($x$)};
\draw (-10,17.5) node [above,rotate=90]  {prey density ($u$)};
\draw [->] (60,0) -- (120,0);
\draw [->] (60,0) -- (60,35);
\draw [-, line width=1.5 pt] (60,15) -- (70,15) node [midway, above, sloped] {\scriptsize  $u(x,t)$}  .. controls (74,15) and (76,30) .. (79,25)  ..controls (80,10) and (90,0.1) .. (100,0.3)..controls (110,0.1) .. (120,0.01);
\draw [-, dashed] (60,28) -- (70,28) node [midway,above,sloped]  {\scriptsize $\phi(t)$} ..controls (80,28) .. (120,28);
\draw [->,line width=0.5] (72,28) -- (72,25);
\draw [->,line width=0.5] (92,28) -- (92,25);
\draw [->,line width=0.5] (112,28) -- (112,25);
\draw (90,0) node [below]  {space ($x$)};
\draw (60,17.5) node [above,rotate=90]  {prey density ($u$)};
\draw (20,35) node [above] {\bf (a) Extinction };
\draw (90,35) node [above] {\bf (b) Uniform extinction};
\end{tikzpicture}
\caption{(a) Sufficient condition for extinction.  The solution $u(t,x)$ is majored by a 
supersolution $\overline{u}(t,x)$. If ${\displaystyle \lim_{t\to+\infty}\overline{u}(t,x)=0}$, 
then ${\displaystyle \lim_{t\to+\infty}u(t,x)=0}$ and there is extinction.\\
(b) Sufficient condition for uniform extinction. The  solution $u(t,x)$ is majored by a 
supersolution $\phi(t)$ which does not depend on $x$. If ${\displaystyle \lim_{t\to +\infty} \phi(t)=0}$, 
then ${\displaystyle \lim_{t\to +\infty} u(t,x)=0}$ uniformly in $x$ and there is uniform extinction.
}
\label{fig:fig1}
\end{figure}
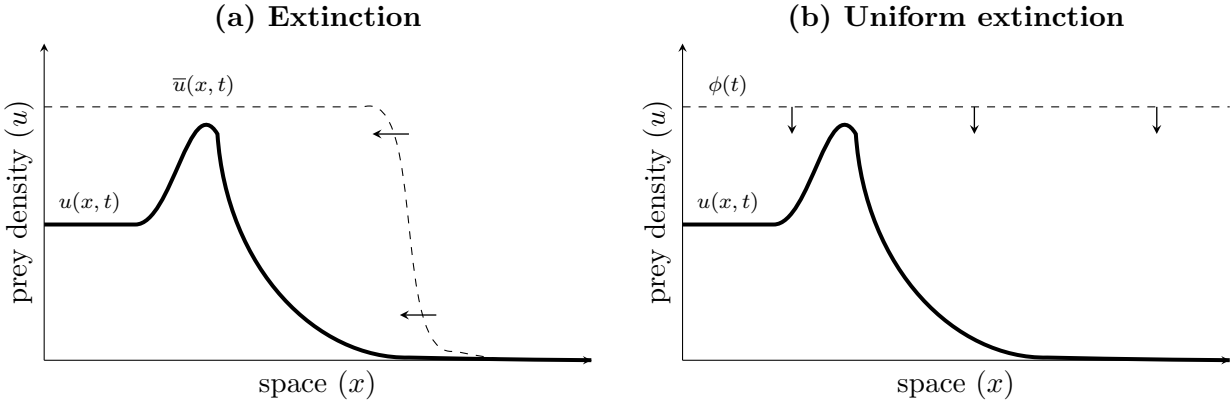
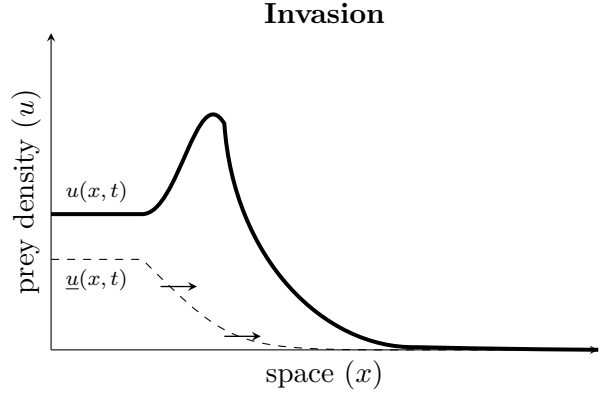
\begin{figure}[htb]
\begin{tikzpicture}[x=0.12cm,y=0.12cm,>=stealth]
\draw(0,0) node {$\,$};
\draw [->] (50,0) -- (110,0);
\draw [->] (50,0) -- (50,35);
\draw [-, line width=1.5 pt] (50,15) -- (60,15) node [midway, above, sloped] {\scriptsize  $u(x,t)$}.. controls (64,15) and (66,30) .. (69,25)  ..controls (70,10) and (80,0.1) .. (90,0.3) ..controls (100,0.1) .. (110,0.01);
\draw [-, dashed] (50,10) -- (60,10) node [midway,below,sloped]  {\scriptsize $\underline{u}(x,t)$} ..controls (70,00) .. (90,0) -- (110,0);
\draw [->,line width=0.5] (62,7) -- (66,7);
\draw [->,line width=0.5] (69,1.5) -- (73,1.5);
\draw (80,0) node [below]  {space ($x$)};
\draw (50,17.5) node [above,rotate=90]  {prey density ($u$)};
\draw (80,35) node [above] {\bf Invasion};
\end{tikzpicture}
\caption{
Sufficient condition for invasion. The  solution $u(t,x)$ is minored by a 
subsolution $\underline{u}(t,x)$. If ${\displaystyle \limsup_{t\to +\infty} \underline{u}(t,x)>0}$, 
then ${\displaystyle \limsup_{t\to +\infty} u(t,x)>0}$ and there is invasion.
}
\label{fig:fig2}
\end{figure}
\begin{theo}\label{thaspcontrol}
Let $E>1$ and define 
$$h_1(E)=\frac{1}{E}\left(2E-1+2\sqrt{E(E-1)}\right).$$
If $h<h_1(E)$ then there is  uniform extinction. 
In other words, for any initial condition verifying \eqref{CI}, there exists $\phi(t)\geq 0$  such that 
any solution $(u(t,x),v(t,x))$ of \eqref{sys2} verifies
 $$\forall\, x\in\R,\; 0\leq u(t,x)\leq \phi(t)\text{ and } \lim_{t\to+\infty} \phi(t)=0.$$
\end{theo}

If $h>h_1(E)$, then there can be invasion or extinction. The following theorem gives a sufficient condition for extinction to occur. 
\begin{theo}\label{thcontrol}
Let $E>1$  be fixed and let $(u,v)$ be the solution of \eqref{sys2}-\eqref{CI}. Define $\underline{v}=1$  and let $\overline{u}$ be a solution of \eqref{eq:soussol} together with  $\overline{u}(0,x)=u(0,x)$. There exists a unique $h^-=h^-(E)>h_1(E)$,  such that 
\begin{itemize}
\item If $h<h^-(E)$, then $\forall x\in \R,$ $\lim\limits_{t\to+\infty} \overline{u}(t,x)=0$
\item If $h>h^-(E)$, then $\forall x\in \R,$ $\lim\limits_{t\to+\infty} \overline{u}(t,x)=\overline{\mu}$ where $\overline{\mu}=\overline{\mu}(E,h)$ is a positive scalar.\end{itemize}
As a consequence,  if $h<h^-(E)$  there is extinction of prey. 
\end{theo}


\noindent The map $E\mapsto h^-(E)$ verifies the following properties proved in section \ref{proofcontrol}.
\begin{properties}\label{prophm}
 The map $E\mapsto h^-(E)$ is increasing and admits the following explicit limits:
 $$\lim_{E\to 1} h^-(E)=1\quad;\quad\lim_{E\to +\infty} h^-(E)=\frac{16}{3}.$$
\end{properties}
\noindent Our last result gives a sufficient condition for invasion to occur.
\begin{theo}\label{thnocontrol}
Let $E>1$ and $\alpha \geq 0$ be fixed and let $(u,v)$ be the solution of \eqref{sys2}-\eqref{CI}. Define $\overline{v}=1+\alpha\frac{E}{1+Eh}$  and let $\underline{u}$ be a solution of \eqref{eq:soussol} together with  $\underline{u}(0,x)=u(0,x)$. There exists a unique $h^+=h^+(E,\alpha)>h^*(E,\alpha)$  such that 
\begin{itemize}
\item If $h<h^+(E,\alpha)$, then $\forall x\in \R,$ $\lim\limits_{t\to+\infty} \underline{u}(t,x)=0$
\item If $h>h^+(E,\alpha)$, then $\forall x\in \R,$ $\lim\limits_{t\to+\infty} \underline{u}(t,x)=\underline{\mu}$ where $\underline{\mu}=\underline{\mu}(E,h)$ is a positive scalar.\end{itemize}
As a consequence,  if $h>h^+(E,\alpha)$   there is invasion of prey. 

\end{theo}

\noindent Finally, the following result specifies the behavior of the map  $h^+$.
\begin{properties}\label{prophp}
The maps $E\mapsto h^+(E,\alpha)$ and $\alpha\mapsto h^+(E,\alpha)$ are increasing. For any $E>1$, $h^+(E,0)= h^-(E)$ and $\lim_{\alpha\to+\infty} h^+(E,\alpha)=+\infty$.
Finally,  $\alpha\geq 0$ being fixed, one has the explicit limit
$$\lim_{E\to +\infty} h^+(E,\alpha)=\frac{8}{3}\left(1+\sqrt{1+\frac{3}{4}\alpha}\right).$$
\end{properties}
\begin{remark}
 The limit ${\displaystyle \lim_{E\to 1} h^+(E,\alpha)}$ remains unknown. 
 However, it can be proved that this limit exists and is greater than ${\displaystyle\lim_{E\to 1} h^*(E,\alpha)}$.
\end{remark}
\noindent The results above are summarized in  figure \ref{fig2}.
In the domain  $\{(E,h),\;E>1,\; h^-(E)<h<h^+(E,\alpha)\}$,  which we will refer to as the ‘transition zone’, 
it is not possible to draw any conclusions concerning whether prey invasion or extinction is likely to occur.
Indeed, this zone can be separated into two subzones, according to the parameters values:
$$\text{Zone I}=\{(E,h),\;E>1,\; max(h^*(E,\alpha),h^-(E) )<h<h^+(E,\alpha) \},$$
$$\text{Zone II}=\{(E,h),\;E>1,\; h^-(E)<h<h^*(E,\alpha)\}.$$
In  Zone I, our numerical simulations show non-monotonous traveling waves. In  zone II,  simulations show various types of behavior, 
including pulse and even heterogeneous positive stationary solutions. This phenomena are discussed in the section \ref{simu}.
\begin{figure}[ht!]
\centering
\begin{tikzpicture}[x=15cm,y=1cm,>=stealth]
\filldraw[draw=white,fill=blue!50] 
 plot[domain=2:1] (\x,0)
 -- plot [domain=1:2,smooth] (\x, {(2-(\x)^(-1)+2*sqrt(1-(\x)^(-1)))})
 -- cycle ;
 \filldraw[draw=white,fill=blue!20] 
    plot [domain=2:1,smooth] (\x, {(2-(\x)^(-1)+(10/3)*sqrt(1-(\x)^(-1)))})
 -- plot [domain=1:2,smooth] (\x, {(2-(\x)^(-1)+2*sqrt(1-(\x)^(-1)))})
 -- cycle ;
  \filldraw[draw=white,fill=red!20] 
    plot [domain=2:1,smooth] (\x, {(2-(\x)^(-1)+(10/3)*sqrt(1-(\x)^(-1)))+2})
 -- plot [domain=3:8,smooth] (1, \x)
 -- plot [domain=1:2,smooth] (\x, 8)
 -- cycle ;
   \filldraw[draw=white,fill=red!40] 
    plot [domain=0:8] (0.9, \x)
 -- plot [domain=8:0] (1, \x)
 -- cycle;
\draw [->] (0.9,0) -- (2,0);
\draw [->] (0.9,0) -- (0.9,8);
\draw [-] (1,0) -- (1,8);
\draw [dashed] plot [domain=1:2,smooth] (\x, {(2-(\x)^(-1)+(10/3)*sqrt(1-(\x)^(-1)))+1})node [right, sloped] { $h^*$};;
\draw plot [domain=1:2,smooth] (\x, {(2-(\x)^(-1)+(10/3)*sqrt(1-(\x)^(-1)))})node [right, sloped] { $h^-$};
\draw plot [domain=1:2,smooth] (\x, {(2-(\x)^(-1)+(10/3)*sqrt(1-(\x)^(-1)))+2})node [right, sloped] { $h^+$};
\draw plot [domain=1:2,smooth] (\x, {(2-(\x)^(-1)+2*sqrt(1-(\x)^(-1)))}) node [right, sloped] { $h_1$};
\draw (1.5,0) node [below] {encounter rate ($E$)};
\draw (0.9,4) node [above, rotate=90] {handling time ($h$)};
\draw (1,0) node [below] {1};
\draw (1.5,1.5) node [above]{ {\LARGE Uniform extinction}};
\draw (1.5,2.6) node [above,rotate=0]{ \LARGE Extinction};
\draw (1.5,6) node [above]{ \LARGE Invasion};
\draw (0.95,4) node [rotate=90]  {\LARGE Uniform invasion};
\draw (1.5,4.5) node [above, rotate=0]{{\large Transition Zone  (I)}};
\draw (1.5,3.4) node [above, rotate=0]{{\large Transition Zone (II)}};
\end{tikzpicture}
\caption{\label{fig2} {Description of the dynamic of the PDE system \eqref{sys2} in the $E - H$ plan.
If $E<1$ there is always invasion. If $E>1$ 
there is a uniform extinction for $h<h_1(E)$, extinction for $h_1(E)<h<h^-(E)$ and 
 invasion for $h^+(E,\alpha)<h$. 
 The zone between  $h^-$ and $h^+$ is called the transition zone. 
This transition zone is splitted into two subzones: Zone I and  Zone II, separated by the $h^*$ curve.  
In these two zones, both extinction or invasion of prey may occur due to various spatial phenomena.}
}
\end{figure}
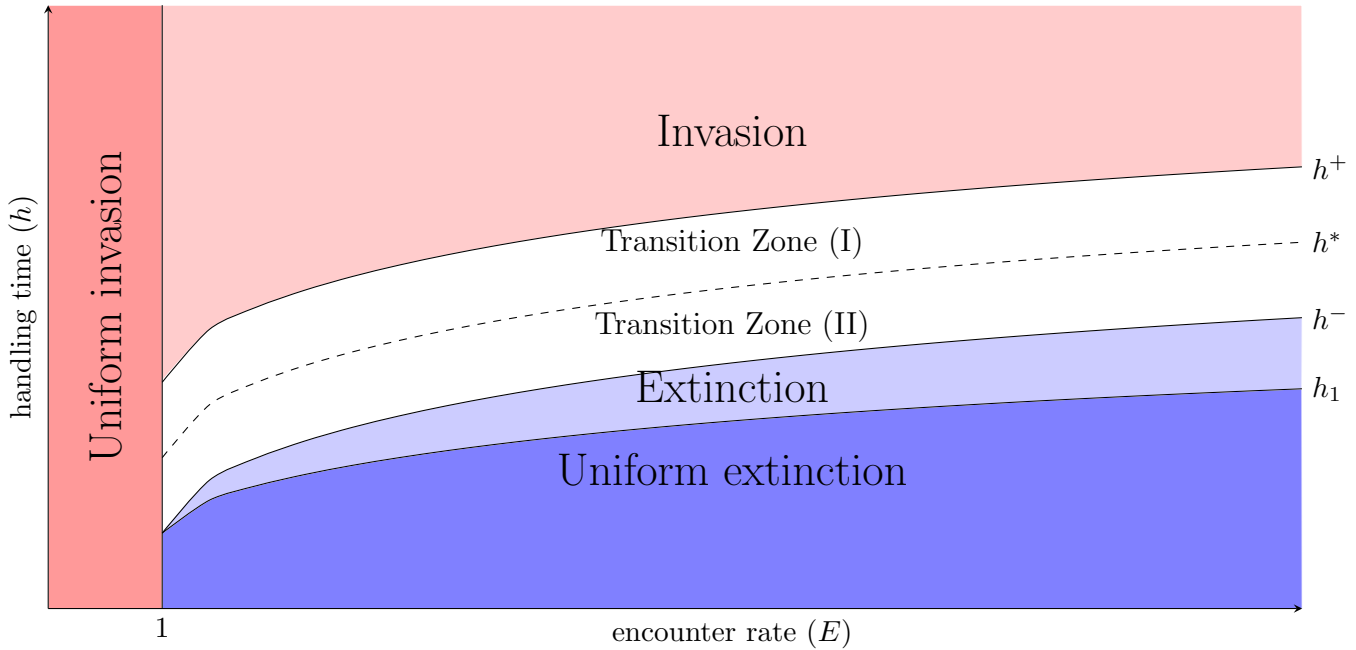
	\newpage

\section{Numerical study of the transition zone }\label{simu}
\subsection{Influence of  $\alpha$.}
The mathematical results above demonstrate the influence of the parameters  $E$ and $h$, and, 
indirectly, that of the conversion rate $\alpha$, on the long-term behavior of the system.
%
%
More precisely, when $E>1$, prey extinction or invasion may occur, depending on the value of $h$.
Indeed, we can define two values  $h^-=h^-(E)<h^+=h^+(E,\alpha)$  (see theorems \ref{thcontrol} and \ref{thnocontrol}).
Extinction occurs if $h<h^{-}$ and invasion occurs if $h>h^+$. When $h\in(h^-,h^+)$ 
we observe richer dynamics, which may depend on other factors. We refer to this zone as the {\bf transition zone}.
Note that, as $h^-$ is not dependent on $\alpha$ and $h^+$ is an increasing function of  $\alpha$ (proposition \ref{prophp}),
the size of this transition zone increases with increasing $\alpha$.\\
\begin{figure}[ht!]

	\centering
		\includegraphics[width=1.1\textwidth]{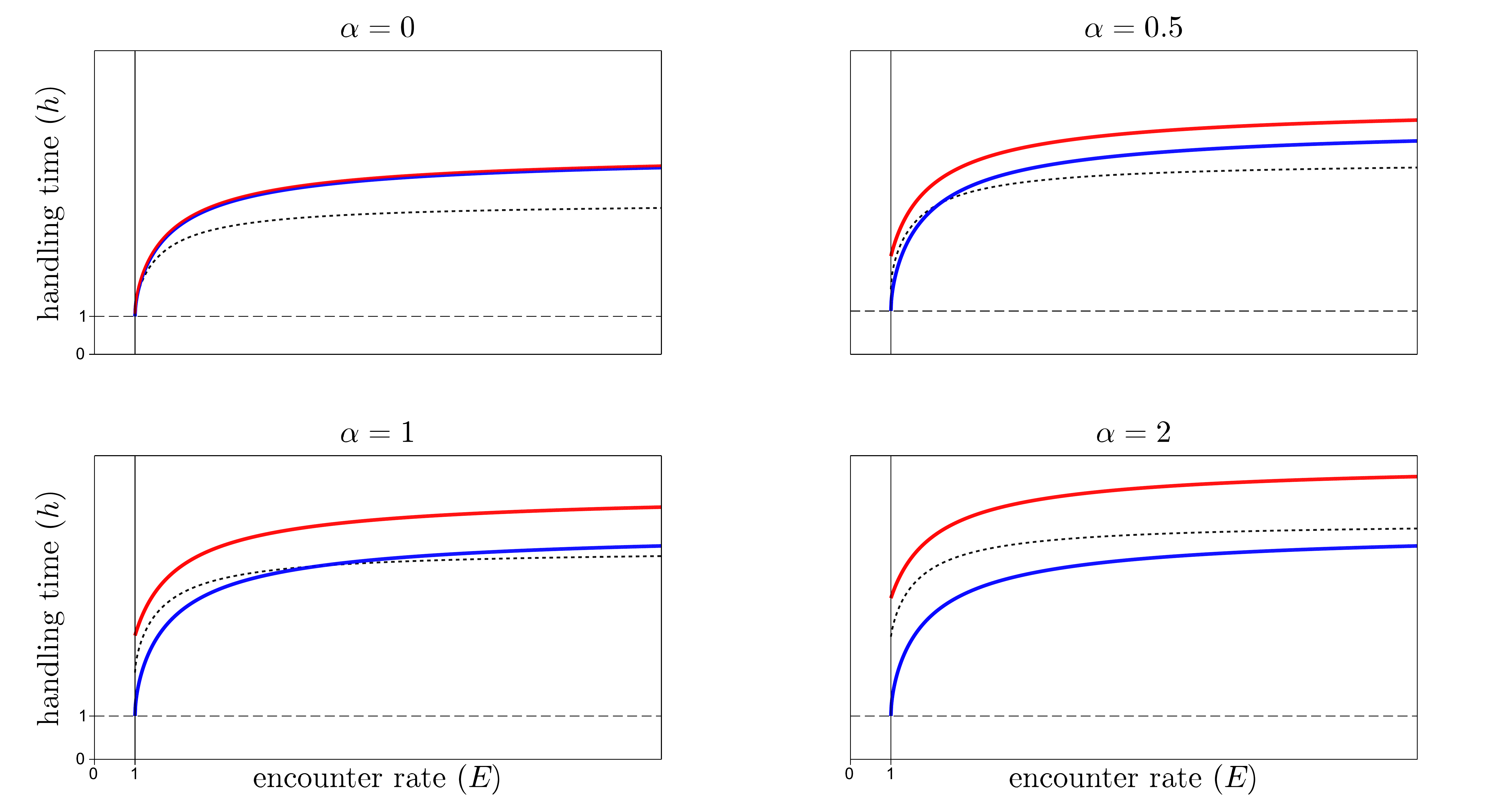}

\caption{\label{figUZ} Computation of $h^*(E,\alpha)$, $h^-(E)$ and $h^+(E,\alpha)$ for four values of $\alpha$. 
The so-called transition zone lies between the red curve $h^+$ and the blue curve $h^-$. 
The transition zone increases with  increasing $\alpha$.  
This transition zone can be split into two subzones separated by $h^*$ (black line).
}

\end{figure}

A first clue to the possible dynamics in the transition zone is provided by an understanding of the dynamics of the ODE system  \eqref{sysODE} 
described in the theorem \ref{thODE}. 
The dynamic of \eqref{sysODE} is essentially dependent  
\footnote{The dynamics generally also depends on a quantity $h^{**}$, defined in the theorem \ref{thODE},
slightly greater than $h^*$ that is not taken into account 
here for the sake of simplicity. }
on the position of $h$ relative to $h^*=h^*(E,\alpha)$. 
When $h<h^*$,  there is no  positive stationary solution, whereas for $h>h^*$
there are two positive stationary solutions, one of which, the larger of the two, is (nearly always) stable.  \\

The position of  $h^*$ relative to $h^-$ and $h^+$ provides a first description of the transition zone. By virtue of proposition \ref{prophet}, 
one gets the following.
We always have $h^*<h^+$ but the position of  $h^*$ relative to  $h^-$ is dependent on $\alpha$.
On the one hand, from the facts that  $h^*(1,\alpha)>h^-(1)$ for $\alpha>0$ and
 $h^*(E,0)= h_1(E)<h^{-}(E)$  for $E>1$, we deduce that $h^->h^*$ for large enough values of $E$ and small enough values of $\alpha$.
 On the other hand, $h^*$ is an increasing function of $\alpha$ tending to $+\infty$. We obtain that $h^*(E,\alpha)>h^-(E)$ 
for  large values of $\alpha$ and any $E>1$.

{\begin{remark}
When $h^->h^*$, which may occur for  sufficiently small values of $\alpha$, 
we see that taking space into account automatically increases the potential of extinction of preys. 
\end{remark}
}

The transition zone can thus be separated into two subzones: one in which  $h<h^*$ (Zone II) and one in which $h>h^*$ (Zone I).
Figure \ref{figUZ} sums up this discussion.
As we will see below, both extinction and invasion are possible in each of these zones, but the phenomena at work differ considerably, 
according to whether $h<h^*$ or $h>h^*$. These phenomena are studied in more detail below, using a numerical approach.

\begin{figure}[ht!]
\includegraphics[width=1.0\textwidth]{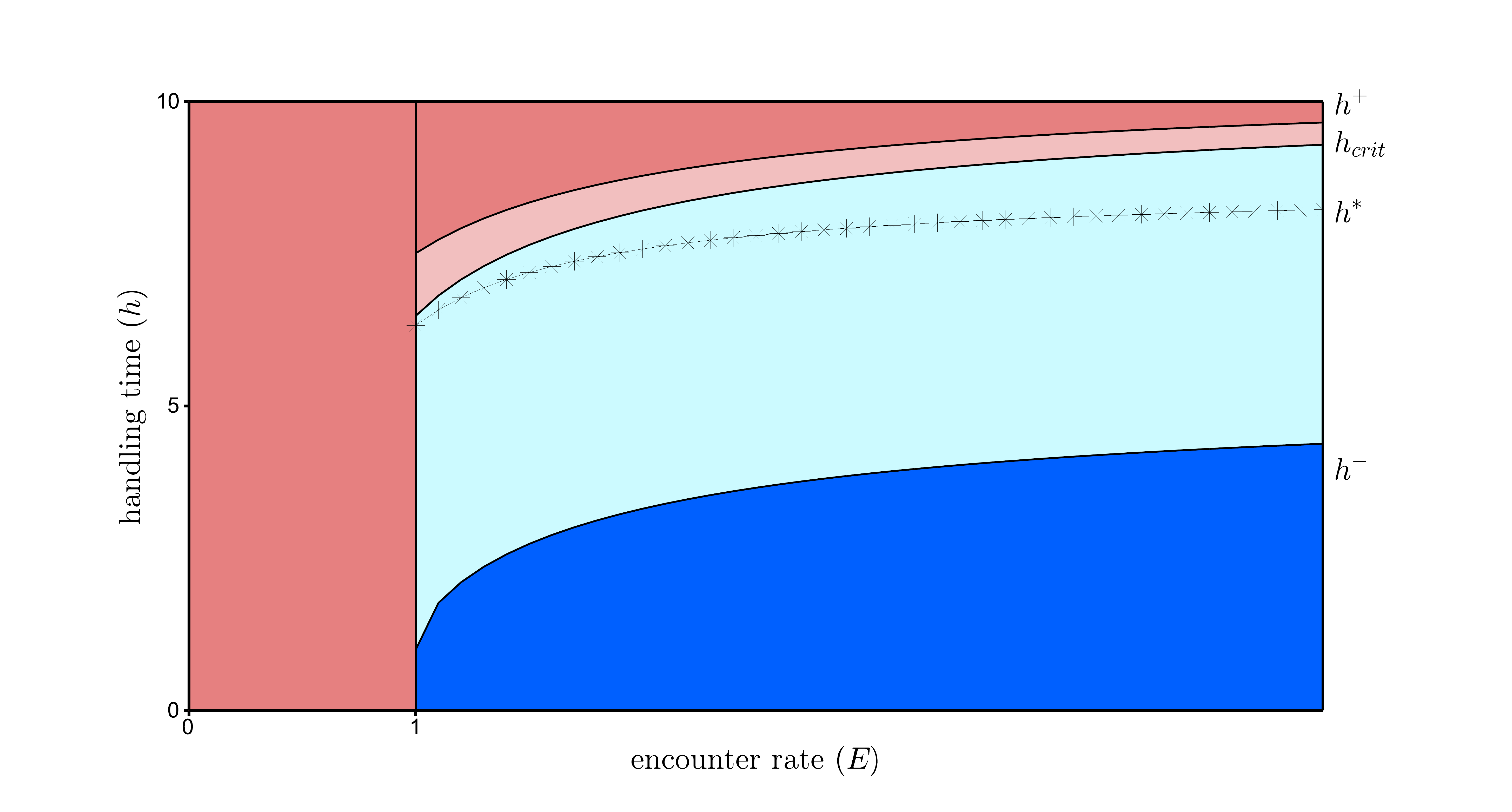}
\caption{\label{fig:hcrit} Numerical computation of $h_{crit}$ in the  $E-h$ space for the fixed values of $\alpha=4$, $d=1$ and $r=1$. 
The transition zone $h\in(h^-,h^+)$ is split into two subzones separated by $h_{crit}$. Extinction occurs below $h_{crit}$ and invasion occurs above.}
\end{figure}

\subsection{Extinction or invasion: influence of $r$ and $d$.}
Our numerical analysis shows that both invasion and extinction are possible in the transition zone. 
When space is taken into account we see that $h^*$ does not separate the zone of invasion from that of extinction. 
These two zones are, indeed, separated by a new critical value of the handling time denoted $h_{crit}\in(h^-,h^+)$, 
 which is dependent on $E$ and $\alpha$, of course, but also on the relative rates of growth ($r$) and diffusion ($d$) of the predator population:
$$h_{crit}=h_{crit}(E,\alpha,r,d).$$

As expected, when $h>h_{crit}$, prey invasion is observed, whereas extinction is observed when $h<h_{crit}$. 
Thus, higher values of $h_{crit}$ are associated with more effective predation and thus with less effective invasion by the prey.

Figure  \ref{fig:hcrit} completes the theoretical scheme represented in figure \ref{fig2}, by presenting an example of the  
curve $(E,h_{crit})$ for a particular selection of values for the parameters 
 $\alpha$, $r$ and $d$  in the $E-h$ plan.

Like $h^*$ and $h^+$, $h_{crit}$ increases with both $E$ and $\alpha$. 
This naturally translates into the fact that,  higher values of $E$ increases the chance of meeting between predators and prey and
that at higher $\alpha$ values, the predator is able to make greater use of the prey and can therefore eliminate it.

Given the multiple dependence of $h_{crit}$ on different parameters, figure \ref{fig:hcrit} can only represent a particular case,
chosen for its simplicity. 
Figure \ref{fighcrit} shows the relationships between  $h_{crit}$ and $d$ for fixed values of $E$ and $\alpha$  and for various values of $r$.
We see that $h_{crit}$ increases with $r$. This translates into the fact that for small values of $r$, the predators growth rate is small and their 
effectiveness reduced. Conversly, $h_{crit}$ (essentially) decreases  with increasing $d$. This is due to the fact that for large $d$, 
predators spread into a zone in which prey are not present, resulting in a weakening predation.
\begin{figure}[ht!]
\hspace{-1cm}\includegraphics[width=1.200\textwidth]{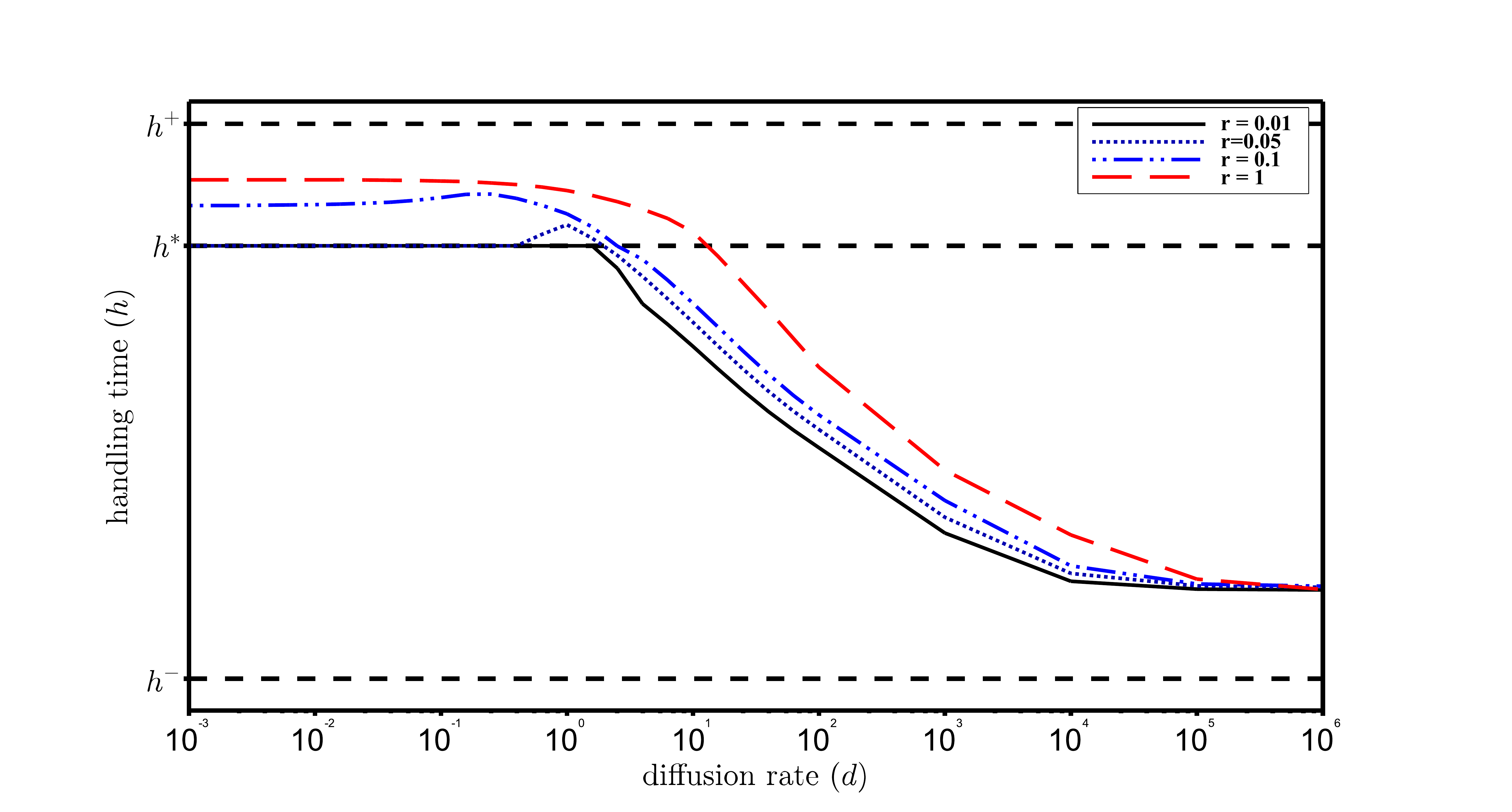}
\caption{\label{fighcrit} {Numerical computation of  $ h_{crit}$ with respect to $d$ for four values of $r$ with $E=2$ and $\alpha=4$.
For a given value of the parameters, there is invasion of prey if $h>h_{crit}$ and extinction if $h<h_{crit}$. Therefore, the larger $h_{crit}$ is, 
the greater is the potential for extinction of prey.
We see that $h_{crit}$ is increasing in  $r$ and (essentially) decreasing in $d$. 
Thus, an increase of $d$ or a decrease of $r$  decrease the impact of the predation on invasive prey.}} 
\end{figure}
\begin{remark}
 For small values of $r$ and intermediate values of $d$,  predators may increase their effectiveness by increasing $d$.
 In that case, 
 the predator growth rate being  small, predators density remains high for a long time even if prey are absent. 
 Now, if $d$ is large enough  but not large, predators may spread into a zone in which 
 prey are not present and remain there at a high density and long enough to stop the prey to invade.
 This phenomenon may enable predators to form a barrier to prey's
 movement, preventing thereby prey propagation. For too  large values of $d$, the loss in predators effectiveness  due to  movement is too strong 
 and the above phenomenon does not hold any more.
 This  explains why, for small values of $r$, $h_{crit}$ first increases  and then decreases with increasing $d$. 
\end{remark}

\subsection{ Dynamics of the system in the transition zone}
We know that extinction occurs when $h<h_{crit}$, whereas invasion occurs when $h>h_{crit}$. 
It should be borne in mind that the existence or absence of non-trivial solutions that are homogeneous over space are dependent on 
the position of $h$ with respect to $h^*$. 
Consequently, the processes at work during extinction or invasion are highly dependent on these position.

We will now describe the different dynamics occurring in the transition zone, summarized in table \ref{bilan}. 
See section \ref{mains} for a precise definition of the various quantities described here.

\begin{table}
	\centering
\begin{tabular}{|c|c|c|}
 \hline&&\\
	& $h<h_{crit}$ : Extinction 	 & $h>h_{crit}$ : Invasion\\
	&  & \\
 \hline
 &&\\
 $h<h^*$: Zone II& 	PULSE ($r\ll 1$)& TURING ($d\gg 1$)  \\
& {\small $h= 5.35$; $d=1$;  $r=0.01$}&     {\small$h= 5.35$; $d=100$;  $r=1$}     \\      
 &\includegraphics[width=0.4\textwidth]{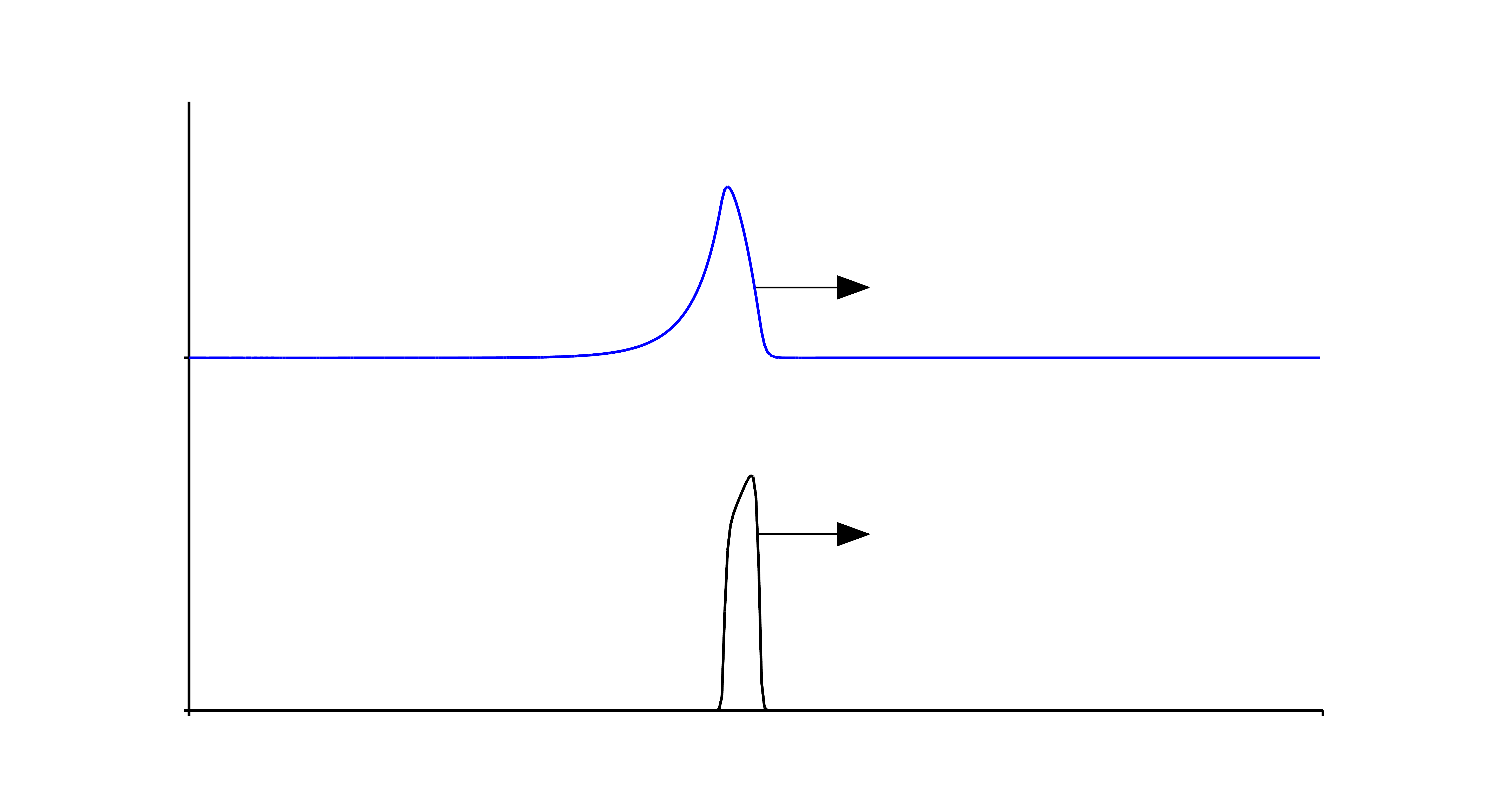}&\includegraphics[width=0.4\textwidth]{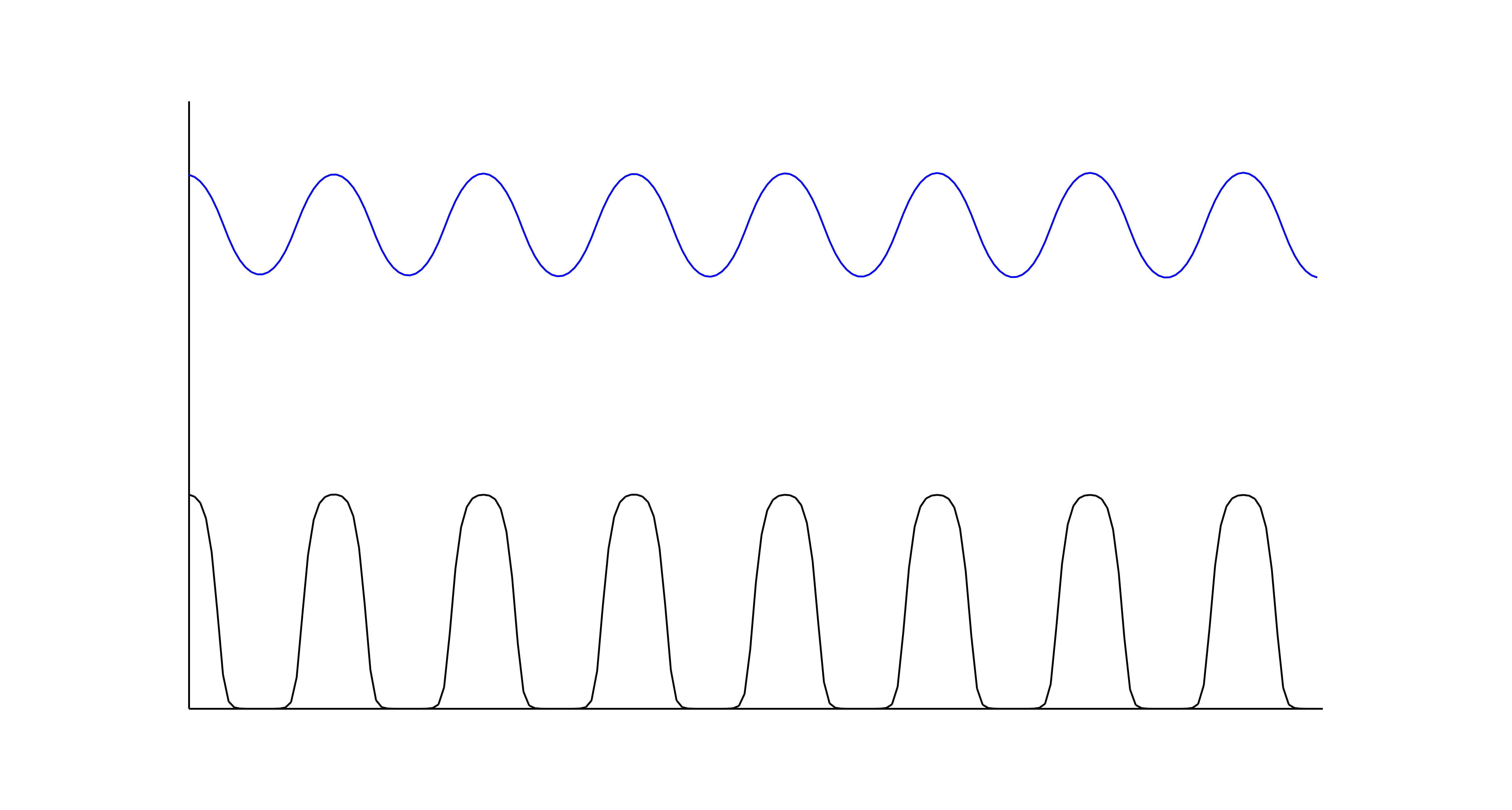}\\                         
 \hline
 &&\\
 $h>h^*$ : Zone I&  ETW		& ITW\\
& {\small$h= 5.6$; $d=1$;  $r=1$}&      {\small$h= 6$; $d=1$;  $r=0.01$}     \\
 &\includegraphics[width=0.4\textwidth]{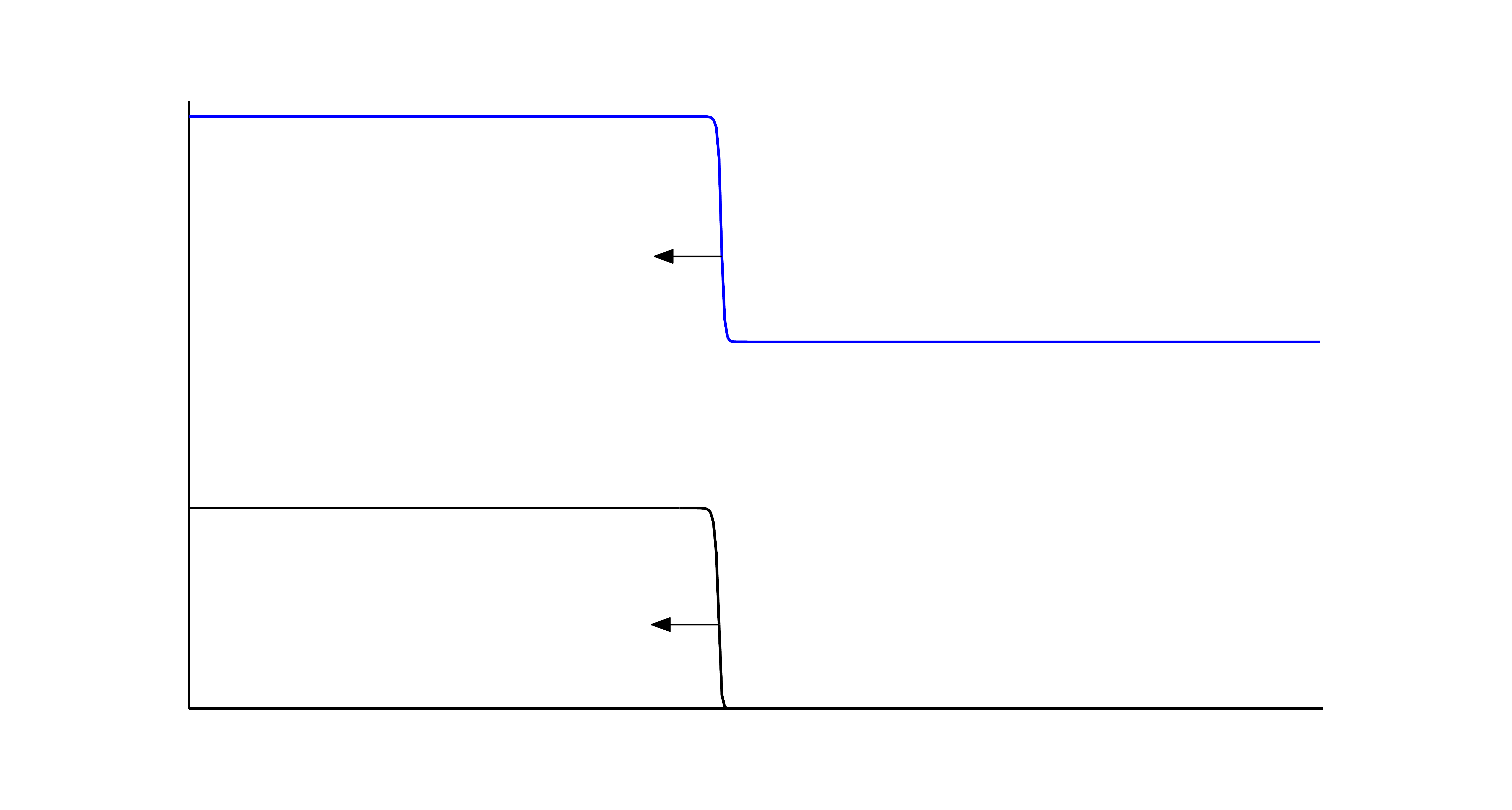}&\includegraphics[width=0.4\textwidth]{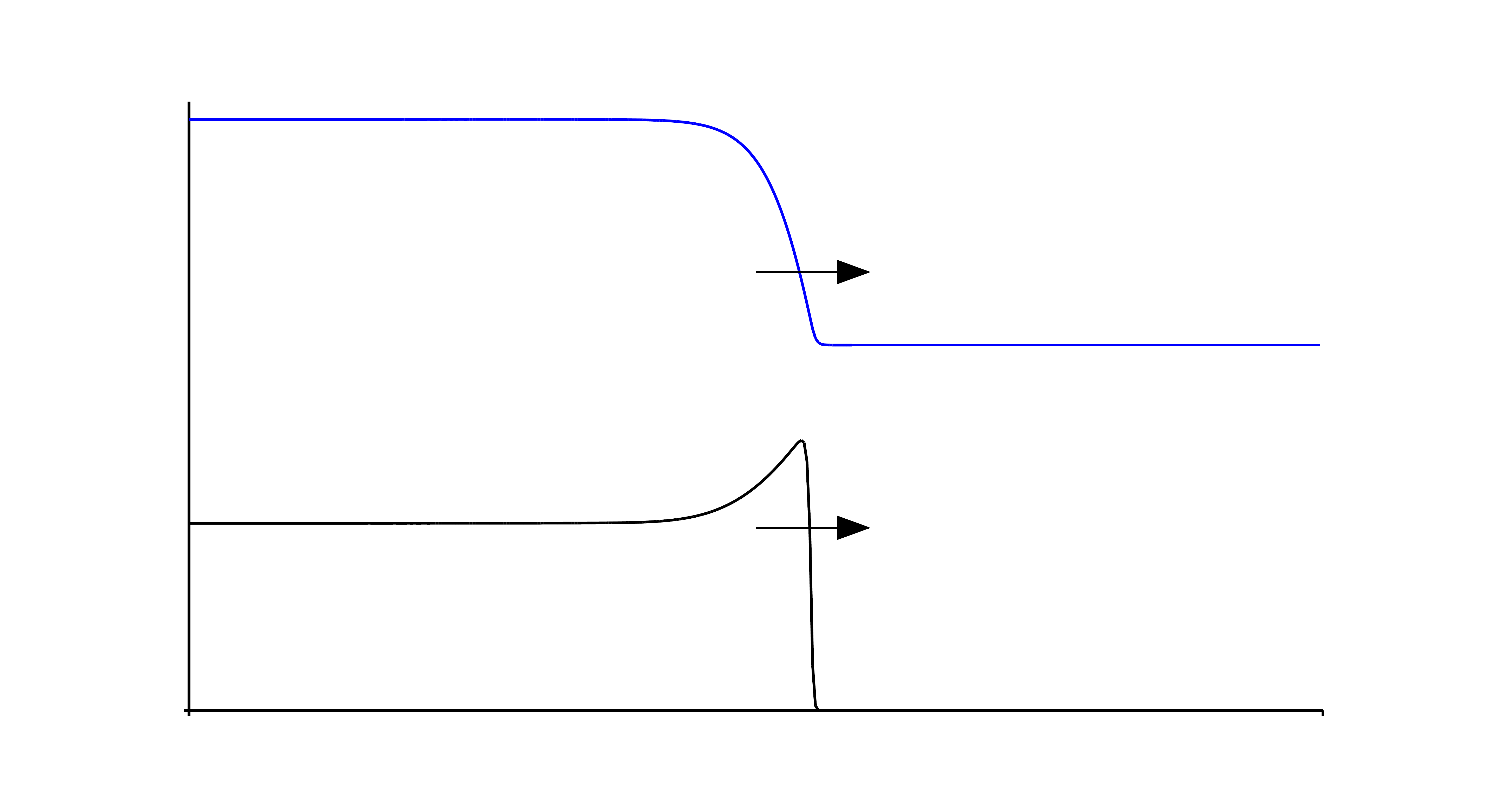}\\
 \hline
\end{tabular}\caption{\label{bilan} {Summary of the four different situations involving space in the transition zone.  Here,  $E=2$ and $\alpha=4$ which yields  $h^*\approx 5.4$.
The X axis represents space and the Y axis the concentration of species. 
The blue curve (top of each graph) represents the concentration of predators in space while the black curve is the concentration of prey.
Extinction and Invasion Traveling Waves are abbreviated ETW and ITW respectively.}}
\end{table}
\paragraph{A] Invasion ($h>h_{crit}$)}
\begin{itemize}
\item {\bf Turing instabilities : $h<h^*$ and $d\gg 1$.}
As $d$ increases,  predators spread out, moving into areas from which the prey is absent,
leading to a decrease in the size of the predator population. 
This phenomenon leads to a decrease in predator density throughout the space occupied by the predator,
allowing the prey to survive in certain zones. We thus obtain a periodic distribution in space and a constant distribution over 
time of the densities of the prey and predator. 
Mathematically, this phenomenon is described by a Turing bifurcation.

\item {\bf Invasion traveling waves  (ITW) : $h>h^*$.} 
The invasion is described simply by an invasion traveling wave: a wave of propagation linking the two stable solutions $(u^*,v^*)$ and $(0,1)$ 
in the direction of the positive solution $(u^*,v^*)$. When  $r$ is large this traveling wave is monotone.
By contrast, when $r\ll 1$  it displays a rich dynamics. When $r\ll 1$ and $d$ is not too large,
we observe that $h_{crit}$ is approximately equal to $h^*$ (see figure \ref{fighcrit}).
This indicates that there is an invasion traveling wave if a positive solution exists. 
In this case, ahead of the front, $v=1$ and, since $r\ll 1$, the predator population increases very slowly. 
Besides, as $h>h^-$, the prey invades the space when the predator is at concentration  $1$.
This leads to  front advancing.
Behind the front, the predator has had sufficient time to increase the size of its population and, therefore, 
to decrease the size of the prey population.
As $h>h^*$,  
the population of the prey decreases towards the positive solution $u^*$ and we observe a non-monotonous invasive traveling wave. 
\end{itemize}

\paragraph{B] Extinction ($h<h_{crit}$)}
\begin{itemize}
\item{\bf Pulse : $h\in [h^-,h^*]$ and $r\ll 1$.} Since $r\ll 1$, 
the  front of the wave  is similar to the ITW described above. However, as $h<h^*$, here is no homogeneous positive solution  $u^*$. 
The prey population therefore decreases to zero behind the front, whereas it continues to advance in ahead of the front. We thus obtain a pulse.
\item {\bf Extinction traveling wave (ETW) : $h>h^*$.} This corresponds to the simplest case described above.
We observe a propagation wave linking the two stable solutions $(u^*,v^*)$ and $(0,1)$ in the direction of the control solution $(0,1)$.
\end{itemize}


Finally,  figure \ref{fig:cartographie} 
presents the map in the $E-h$ plane for fixed values of $\alpha$, $d$ and $r$. It furthermore  specifies the possible dynamics  in each zone.

\begin{figure}[ht]
\includegraphics[width=1\textwidth]{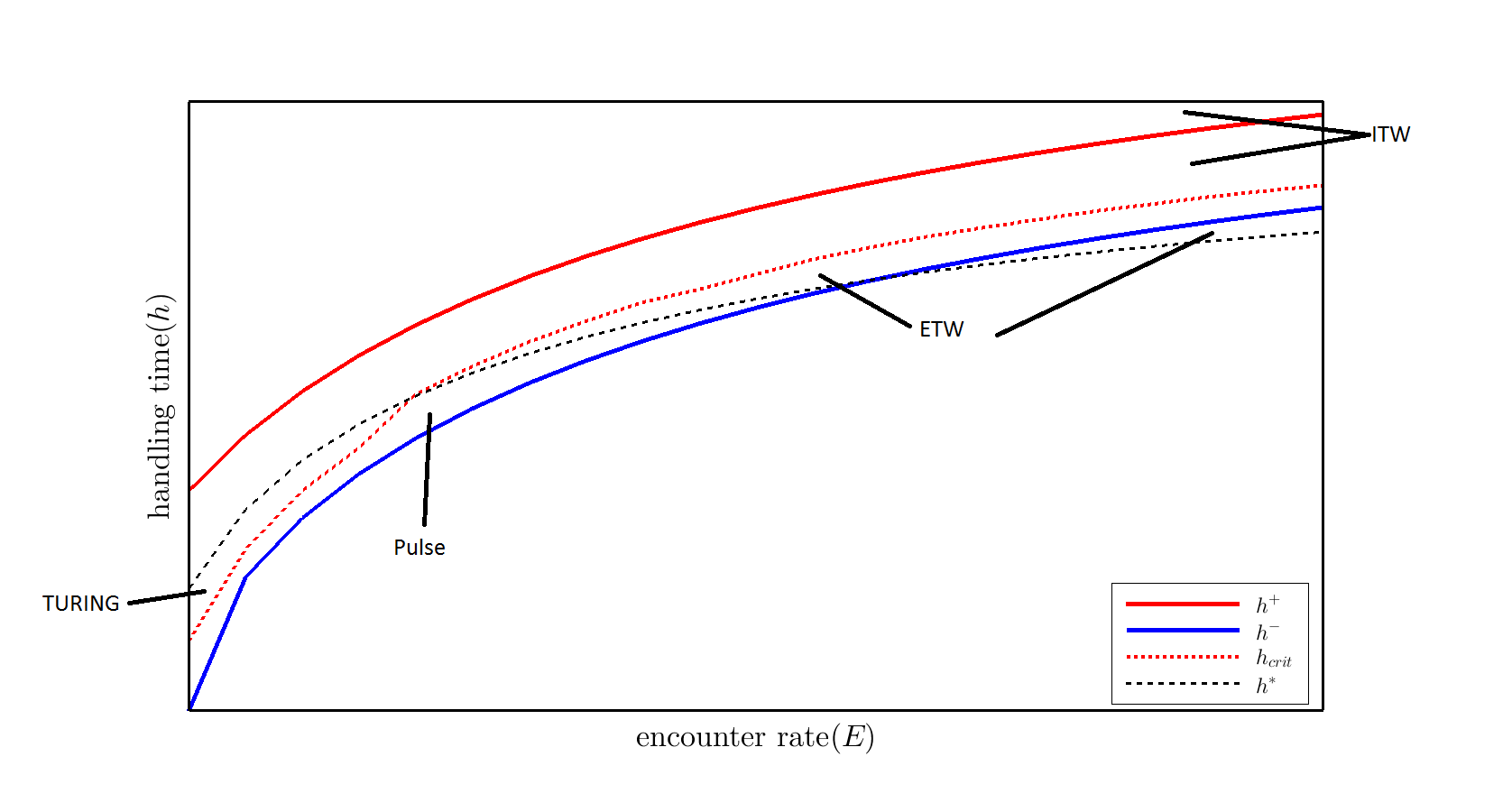}
\caption{\label{fig:cartographie}
{Different spatial dynamics in the transition zone. To ensure that all possible situations are represented, 
we choose  $\alpha=0.5$, $r=0.01$ and $d=10$. 
Note that traveling waves of extinction (ETW) and traveling waves of invasion (ITW) may be obtained outside the transition zone.}}
\end{figure}
\section{Conclusion and discussion}\label{ccl}
\subsection{Summary of the results}
We have shown that invasion occurs if $E<1$. 
If  predators do not encounter their prey they cannot control them.
%
If $E>1$, then extinction or invasion can occur, depending on the parameters $h$, $E$, $\alpha$ and $r$.
Uniform extinction occurs for $h<h_1(E)$ and extinction occurs for $h_1(E)<h<h^-(E)$.
Invasion occurs for $h^+(E,\alpha)<h$.
Thus, if $h$ increases, we move from a zone of extinction without a consideration of space to a zone of extinction requiring a consideration 
of spatial aspects and then to a zone of invasion. 
%
For intermediate values of  $E$, 
the zones of extinction increase with increasing $E$ resulting in a higher potential of extinction, as $h_1,h^*, h^-$  are
increasing functions of $E$.  
When $E$ is large, the zones of control do not depend  on $E$ any more because  $h_1,h^-, h^+$ have finite limits when  $E \rightarrow +\infty$.
Thus, $E$ can only play a role in prey control if it takes intermediate values and if $h$ is not too large. 
In summary, for low values of $E$ ($E<1$) or high values of $E$ or $h$, 
the outcome of the interaction (extinction or invasion) is independent of $E$.\\

There is furthermore a transition zone splitted in two subzones, 
with various spatio-temporal phenomena and wherein both extinction and invasion can occur. 
The size of this transition zone greatly increases when the conversion rate $\alpha$ increases.
Depending on the relative positions of these two zones with regard to the zones of extinction and of invasion,
four spatial dynamics were identified: extinction and invasion traveling waves, extinction pulse waves and heterogeneous stationary 
positive solutions of the Turing type. 

\subsection{Biological interpretation of the main results}
We have shown that an increase in $E$ increases the potential of extinction while an increase in $h$ increases the potential of invasion. 
This translates the fact that a highly effective predator does have a high encounter rate and a small handling time.
Furthermore, since $h^+$ is an increasing function of $\alpha$, an increase in $\alpha$ decreases the potential of invasion and increases the size of 
the transition zone, which in turn increases the potential for  the system to have  complex dynamics. 
Finally, an increase of the diffusion rate $d$  and a decrease of the amplitude of the predators growth rate $r$  both 
increase the potential of invasion of prey. 
Thus, a generalist predator loses its effectiveness to exterminate invasive prey if it diffuses  too fast or if it has a too slow dynamics.

The above results are stated in term of adimensionalized parameters (see section \ref{secadim}).
By choosing the appropriate spatio-temporal variables, we may define $D_u=r_1=1$. Thus, the biological interpretations of $d$ and $r$ are accurate.
Conversely, the definitions of the searching efficiency, $E=E K_2$, 
the handling time $h= h \frac{K_1}{K_2}$ and the conversion rate $\alpha=\frac{\gamma}{r}$ complicate the biological 
interpretation of these three parameters. Thus, in addition to the above discussion about the influence of $E$ and $h$,
we now discuss our results in terms of the  other biological variables:  $K_1$, $K_2$, $\gamma$ and $r$.
The parameter $h$ being increasing in the carrying capacity $K_1$ of prey, prey with high carrying capacity show a high risk of being invasive.
Conversly, $h$ and $E$ are respectively decreasing and increasing in the carrying capacity of predators $K_2$. Thus, predators with high carrying capacity
have a high potential to control prey invasion. 
Otherwise  $\gamma<1$ means that predator growth is mostly due to alternative prey while  $\gamma>1$  implies that 
predator growth is due to consumption of the focal invasive prey. 
Finally, an increase in $\gamma$ and a decrease in the amplitude of the predator growth rate $r$   yield an increase of $\alpha$. 
Therefore,  predators with
a preference for  the invasive prey or predators with a slow dynamics might display a complex dynamics. In particular, the likelihood of the system to 
exhibit a pulse wave is then important.

\subsection{The consideration of space often, but not always, increases the potential for control of pest invasion}

The model analyzed here was studied without taking space into account, except for a numerical exploration in the discussion, 
in an article by \cite{Magal}.
As explained in the introduction, models of identical structure have been proposed independently by \cite{Fagan2002} and 
\cite{Chakraborty}. 
We will now discuss our results in the context of these previous studies. 
Adding a spatial component to predator-prey systems makes any prediction about the controllability of the system difficult, 
as it then depends on the values of several parameters. The comparison between situations with and without the consideration of space is 
epitomized by the distinction between $h^*$, separating parameter regions of mono- and bistability in the ODE system, and $h_{crit}$, separating parameter 
regions of invasion and extinction in the PDE system. 
We will now  focus on the case of $E>1$, as values of $E<1$ do not promote control,  predators encountering  prey too infrequently.

If space is not taken into account, control occurs if $h<h^*$, as $0$ is a global attractor. 
This is still true in situations in which space is taken into account, if  $h<h_1$ where $h_1$ is smaller than $h^*$.
When $h$ is between $ h_1$ and $h^*$, $0$ is only a local attractor, so it is not possible to state that control is always attained.
In this respect, adding  consideration of space decreases the potential for control. 
Furthermore, when space is not taken into account, there is either extinction or 
invasion when $h>h^*(E,\alpha)$, {depending on initial conditions}.
Incorporating  consideration of space changes the region where invasion occurs, for any values of the other parameters and for appropriate initial conditions,  into   $h>h^+(E,\alpha)$, with $h^+>h^*$. 
Thus, the consideration of space reduces the size of the  zone wherein the invasion is certain and is detrimental to the invading prey. 
Finally, the relative levels of predator and prey diffusion also determine the potential for control. 
Our model shows that control is increased by predators being less mobile than prey. If predator mobility levels are too high,
the predators become to thinly spread on the ground. For similar reasons, too high a level of prey mobility leaves the prey vulnerable to predators. 
This is entirely consistent with the experimental findings of \cite{Fagan2002}. In conclusion, taking space into account can lead to an increase or
a decrease in the controllability of invading prey by predators; 
the addition of space to the model has no generic implication for considerations of predator-prey dynamics 
{(see also \cite{LamNi,Braverman2015})}.

\subsection{How can generalist predators  reverse invasion by pest?}

The originality of this study lies in its consideration of a generalist predator in a spatial context. When studying generalist predators, 
it is common practice to assume that the functional response is of type III, due to switching between prey species (\cite{Ebarch2014, Leuven2007, Leuven2013, Petro2013}). 
However, this approach is not mandatory, and other works (\cite{Basnet,Krivan2006,Hoyle}) 
have considered a type II functional response. Altering our model to include a type III functional response would be very costly 
{in terms of understanding}, 
because such responses lead to a  loss of bistability. Its derivative would be null without prey, so some of our demonstration would fail 
and the analytical complexity would be greatly increased. However, traveling waves for {\it specialist} predator with type III functional response 
are known to exist (\cite{LiWu2006}) which indicates that our result may be extended to this case.

The complexity of analytical studies of spatial predator-prey interactions lies in the reaction terms being of alternative signs in the equations, 
making the study of the {\it systems} of equations essential (\cite{Dunbar2,Huang,HuangWeng}). Other interactions, such as competition of two species (all negative) and symbiosis (all positive), 
are simpler, as their studies are similar to the study of a single equation (\cite{volperts,Alzahrani2012}). This accounts for the slow scientific progress in this otherwise highly 
relevant topic. However, several major results have been obtained in recent decades, including those of the fundamental work of  \cite{owenlewis}.
The finding of Owen and Lewis  that predators can slow, stop, and even reverse invasion by their prey was based on the bistability of the 
prey-only dynamics of systems consisting of specialist predators attacking prey populations displaying Allee effects. By contrast, our work shows 
that the ability of generalist predators to control prey populations with logistic growth lies in the bistable dynamics of the {\it coupled} system. 
We also observe pseudo-Allee effects in our system, but their physics is quite different. An analysis of the ODE system identified parameter
regions of monostable (extinction) and bistable (extinction or invasion) dynamics, but analysis of the associated PDE was able to distinguish different 
and additional regions of invasion and extinction. As a consequence, prey control was predicted to be possible when space was considered in additional 
situations other than those identified without considering space. The reverse situation was also possible. None of these considerations apply to spatial
predator-prey systems with specialist natural enemies.

\section{Proofs}\label{Proofs}

\subsection{Proof of theorem \eqref{thODE}}\label{proofode}
\noindent Let $E>1$ and $\alpha\geq 0$ be fixed. For any $h\geq 0$, the system \eqref{sysODE} can be rewriten as
 \begin{equation}\label{sysODEproof}
  \left\{\begin{array}{l}
   \frac{d}{dt} u= \Theta_h(u) (f_h(u)-v)\\
   \frac{d}{dt} v= r v(g_h(u)-v)
  \end{array}\right.
 \end{equation}
where $\Theta_h(u)=\frac{E u}{1+Ehu}$, $f_h(u)=\frac{1}{E}(1-u)(1+Ehu)$ and $g_h(u)=1+\alpha\frac{Eu}{1+Ehu}$.\\

\noindent{\it (i) Proof of the existence.}
 Define $H(h,u)=f_h(u)-g_h(u)$. For a given $h\geq0$,  a couple $(u,v)$ is a positive  stationary solution of \eqref{sysODEproof} if and only if $v=f_h(u)$ and
 \begin{equation}\label{solODE}
 u\in]0,1[ \text{ is a solution of }H(h,u)=0.
\end{equation}
 Now, fix $u\in]0,1[$. Since $\partial_h H(h,u)=u(1-u)+\alpha \left(\frac{Eu}{1+Ehu}\right)^2>0$, one sees that the map $h\mapsto H(h,u)$ is increasing. 
 From $E>1$, we get $H(0,u)<0$ and from $u\in]0,1[$  we get $\lim\limits_{h\to +\infty} H(h,u)=+\infty.$
The map $h\to H(h,u)$ being continuous, this implies that for any $u\in]0,1[$, there exists a unique $h(u)>0$ such that 
$\begin{cases}
H(h,u)<0\text{ if }h<h(u),\\ 
H(h,u)=0\text{ if }h=h(u),\\
H(h,u)>0\text{ if }h>h(u).\\
\end{cases}$
\begin{figure}[h]
\centering
\begin{tikzpicture}
\newcommand\E{2};
\newcommand\al{2};
\begin{scope} [x=5cm,y=0.2cm,>=stealth]
\newcommand\ymax{50};
\clip(-0.2,-5) rectangle (1.5,\ymax+1);
\draw[->] (-0.01,0)--(1.1,0) ;
\draw[->] (0,-0.5)--(0,\ymax);
\draw[-,dashed] (0,13.2)--(1,13.2);
\draw[-,dashed] (0.39,13.2)--(0.39,0) node[below]{$u_{crit}$};
\draw[-,dashed] (1,0)--(1,\ymax);
 \filldraw[fill=gray!20,line width=2pt] plot[domain=0.03:0.99,samples = 1000] (1-\x,{((\x*(1-\x))^(-1))*(0.5*(1+sqrt(1+4*\al \x*(1-\x))))-(\E*\x)^(-1)});
\draw (1.1,0) node[below]{u};
\draw (0,\ymax) node[left]{h};
\draw (0,0) node[below]{0};
\draw (1,0) node[below]{1};
\draw (0,0) node[left]{0};
\draw (0,13.2) node[left]{$h^*$};
\draw (0.15,35) rectangle (0.75,45);
\draw (0.45,40) node{$f_h(u)>g_h(u)$};
\filldraw[color=white] (0.15,02) rectangle (0.75,12);
\draw (0.15,02) rectangle (0.75,12);
\draw (0.45,7) node{$f_h(u)<g_h(u)$};
\end{scope}
\begin{scope}[x=3cm,y=1.5cm,>=stealth,xshift=7cm,yshift=7.5cm]
\newcommand\h{5};
\draw (-0.2,-0.3) rectangle (1.4,2);
\draw (0.6,2) node[below]{$h>h^*$};
\draw (0.3,1.7) rectangle (0.8,2);
\draw [->] (0,0) -- (1.2,0);
\draw [->] (0,0) -- (0,1.5);
\draw  plot[domain=0:1,smooth] (\x, {((\E)^(-1))*(1-\x)*(1+\E*\h*\x)});
\draw plot [domain=0:1.2,smooth] (\x, {1+(\al*\E*\x) / (1+\E*\h*\x)}) node [below, sloped] { $g_h(u)$};
\draw (1,0) node [below] {\footnotesize 1};
\draw (0,0) node [below] {\footnotesize 0};
\draw (0,0) node [left] {\footnotesize 0};
\draw (0,1) node [left] {\footnotesize 1};
\draw (0,1/2) node [left] {\footnotesize $\frac{1}{E}$};
\draw (1.2,0) node [below] {\footnotesize $u$};
\draw (0,1.5) node [left] {\footnotesize $v$};
\draw (1.1,0.2) node {\footnotesize  $f_h(u)$};
\end{scope}
\begin{scope}[x=3cm,y=1.5cm,>=stealth,xshift=7cm,yshift=4cm]
\newcommand\h{4.36};
\draw (-0.2,-0.3) rectangle (1.4,2);
\draw (0.6,2) node[below]{$h=h^*$};
\draw (0.3,1.7) rectangle (0.8,2);
\draw [->] (0,0) -- (1.2,0);
\draw [->] (0,0) -- (0,1.5);
\draw  plot[domain=0:1,smooth] (\x, {((\E)^(-1))*(1-\x)*(1+\E*\h*\x)});
\draw plot [domain=0:1.2,smooth] (\x, {1+(\al*\E*\x) / (1+\E*\h*\x)}) node [below, sloped] { $g_h(u)$};
\draw (1,0) node [below] {\footnotesize 1};
\draw (0,0) node [below] {\footnotesize 0};
\draw (0,0) node [left] {\footnotesize 0};
\draw (0,1) node [left] {\footnotesize 1};
\draw (0,1/2) node [left] {\footnotesize $\frac{1}{E}$};
\draw (1.2,0) node [below] {\footnotesize $u$};
\draw (0,1.5) node [left] {\footnotesize $v$};
\draw (1.1,0.2) node {\footnotesize  $f_h(u)$};
\end{scope}
\begin{scope}[x=3cm,y=1.5cm,>=stealth,xshift=7cm,yshift=0.5cm]
\newcommand\h{3};
\draw (-0.2,-0.3) rectangle (1.4,2);
\draw (0.6,2) node[below]{$h<h^*$};
\draw (0.3,1.7) rectangle (0.8,2);
\draw [->] (0,0) -- (1.2,0);
\draw [->] (0,0) -- (0,1.5);
\draw  plot[domain=0:1,smooth] (\x, {((\E)^(-1))*(1-\x)*(1+\E*\h*\x)});
\draw plot [domain=0:1.2,smooth] (\x, {1+(\al*\E*\x) / (1+\E*\h*\x)}) node [below, sloped] { $g_h(u)$};
\draw (1,0) node [below] {\footnotesize 1};
\draw (0,0) node [below] {\footnotesize 0};
\draw (0,0) node [left] {\footnotesize 0};
\draw (0,1) node [left] {\footnotesize 1};
\draw (0,1/2) node [left] {\footnotesize $\frac{1}{E}$};
\draw (1.2,0) node [below] {\footnotesize $u$};
\draw (0,1.5) node [left] {\footnotesize $v$};
\draw (1.1,0.2) node {\footnotesize  $f_h(u)$};
\end{scope}
\end{tikzpicture}
\caption{\label{fig:h(u)} The four figures are computed for $E=\alpha=2$, which gives $h^*\approx4.36$. 
The figure on the left represents the curve $u\mapsto h(u)$ (in bold). Above the curve $f_h(u)>g_h(u)$ 
while below the curve $f_h(u)<g_h(u)$. For a given $h$, the ordinate $u$ of a point of this curve verifies $f_h(u)=g_h(u)$ and  corresponds to the positive stationary solution  $(u,f_h(u))$ of \eqref{sysODEproof}.
The figures on the right represent the isoclines $v=f_h(u)$ and $v=g_h(u)$ 
for the three  fixed  values  $h=3$, $h=4.36$ and  $h=5$. A positive stationary solution of the system \eqref{sysODEproof} 
corresponds to an intersection of these two isoclines. 
The system \eqref{sysODEproof} admits two positive solutions for $h>h^*$, one (double) 
solution for the critical case $h=h^*$ and zero positive solution for $h<h^*$. }
\end{figure}
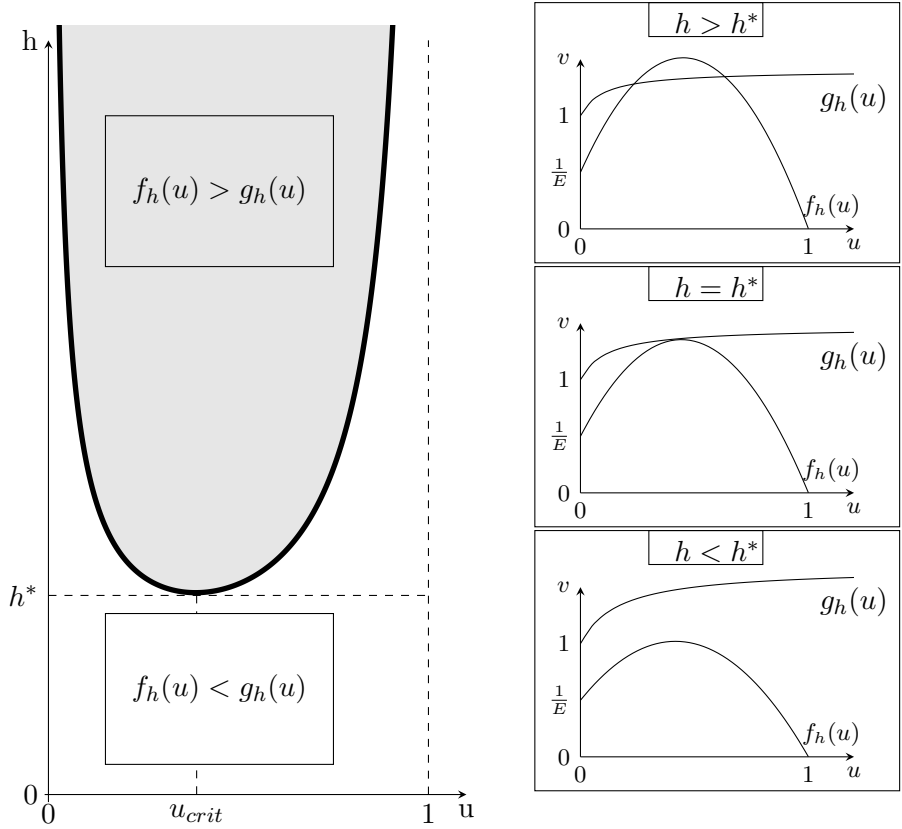

\noindent The smooth function $u\mapsto h(u)$ may be computed explicitly  
by noting  that for any $h\geq 0$, the equation \eqref{solODE} is equivalent to the algebraic equation 
\begin{equation}\label{solODEpol}
 u\in]0,1[ \text{ is a solution of }P_h(u)=0\end{equation}
wherein we have set
$$P_h(u)=(1-u)(1+Ehu)^2-E(1+Ehu+ E\alpha u).$$ 
\noindent This yields  the explicit formula
\begin{equation}\label{formulah}
h(u)=\frac{1}{Eu(1-u)} \left[ \frac{E}{2}\left(1+\sqrt{1+4\alpha u(1-u)}\right)+u-1\right].
\end{equation}
In particular
\begin{equation}\label{limitsh}\lim_{u\to 0^+} h(u)=\lim_{u\to 1^-} h(u)=+\infty.\end{equation}
 This implies that the minimum of $h(u)$ is obtained for some $u_{crit}\in]0,1[$. We define
\begin{equation}\label{defhstar}h^*=\inf_{u\in(0,1)} h(u)=\min_{u\in(0,1)} h(u)=h(u_{crit}).
\end{equation}
The definition of $h^*$  shows that if $h<h^*$, then \eqref{solODEpol} has zero solution.
This also implies, together with  the limits \eqref{limitsh} and the continuity of
$u\mapsto h(u)$,  that for any fixed  $h>h^*$ 
the equation  \eqref{solODEpol} admits at least two solutions\footnote{\label{minhstar}Remark that  $h^*>\frac{1}{E}$, because  $h\leq \frac{1}{E}$, implies that $f_h'<0$ on $(0,1)$ and \eqref{solODE} has no solution.} (see the figure \ref{fig:h(u)}).
In addition to this, for any fixed $h>0$, one has $\deg(P_h)=3$  and  $P_h$ always admits  a negative roots for $P(0)=1-E<0$ and $\lim\limits_{x\to -\infty}P_h(x)=+\infty$.  
This implies that \eqref{solODEpol} admits at most two solutions\footnote
	{				\label{footproofhstar}
				These arguments also show that $u\mapsto h(u)$ is decreasing on $(0,u^*)$ and increasing on $(u^*,1)$ ; for otherwise  it is possible
				to choose  $h>0$ such that there is at least four different $u\in(0,1)$ such that $h=h(u)$, which is equivalent to $P_h$ having at least 4 roots. See the figure \ref{fig:h(u)}.
	}.
In conclusion, \eqref{solODEpol} has exactly two positive solutions if $h>h^*$ and zero positive solution if $h<h^*$.  This ends the proof of $(i)$.\\

 \noindent{\it (ii) Proof of the stability}. Let $h>h^*$ be fixed. Let  $(u,v)$ be a positive stationary solution  of \eqref{sysODEproof}. Since $(u,v)$ verifies $f_h(u)=g_h(u)=v$, the Jacobian matrix at $(u,v)$ reads
 
$$J(u,v)=\left[\begin{array}{cc}
                    \Theta_h(u)f_h'(u)&-\Theta_h(u)\\ rvg_h'(u)&-rv
                   \end{array}\right]$$
hence
 $$\det\big(J(u,v)\big)=r\Theta_h(u)v\big(g_h'(u)-f_h'(u)\big).$$
From the proof of (i), we know that the system \eqref{sysODEproof} admits exactly two positive solutions denoted respectively as $(\widehat{u},\widehat{v})$ and $(u^*,v^*)$ with $\widehat{u}<u_{crit}<u^*$ and such that 
\begin{equation}\label{equalh}
h=h(\widehat{u})=h(u^*),
\end{equation}where $u\mapsto h(u)$ is given by \eqref{formulah}. In particular, $\widehat{u}$ and $u^*$ are the solution of 
\begin{equation}\label{H(u)}
H(h(u),u)=0.
\end{equation}
Differentiating the equation \eqref{H(u)} with respect to $u$ gives
$$\partial_u H(h(u),u)=-h'(u)\partial _h H(h(u),u).$$
Thus, using the known fact that $\partial _h H(h(u),u)>0$, the identity \eqref{equalh} and the footnote \ref{footproofhstar}, one gets $\partial_u H(h,\widehat{u})<0$ and $\partial_u H(h,u^*)>0$.
Since $\partial_u H(h,u)=-(g'_h(u)-f_h'(u))$, this shows that $\det\big(J(\widehat{u},\widehat{v})\big)<0$ and the instability of $(\widehat{u},\widehat{v})$ follows.\\

\noindent By contrast, one has $\det\big(J(u^*,v^*)\big)>0$ and it appears that the stability of $(u^*,v^*)$ is given by the sign of
\begin{equation}\label{trace}\text{tr}\big(J(u^*,v^*)\big)= \Theta_h(u^*)f_h'(u^*)-rv^*.\end{equation}
In order to highlight the dependence on $h$, for any $h>h^*$, we note $u^*=u^*(h)$ and we also define $\mu(h)=\frac{Eh-1}{2Eh}$. 
From $f'_h(u)=\frac{1}{2E^2 h}(\mu(h)-u)$ and \eqref{trace}, we infer the following:
\begin{itemize}
		\item If $\mu(h)\leq u^*(h)$ then $(u^*,v^*)$ is asymptotically stable.
		\item If $\mu(h)> u^*(h)$ then the stability of $(u^*,v^*)$ depends on $r$.
					More precisely, define $r_{crit}=\frac{\Theta_h(u^*)f_h'(u^*)}{v^*}>0.$ 					(It is easy to show that $r_{crit}\leq 1$).
					\begin{itemize}
							\item If $r>r_{crit}$, then  $(u^*,v^*)$ is asymptotically stable.
							\item  If $r<r_{crit}$, then $(u^*,v^*)$ is unstable.
							\end{itemize}

\end{itemize}
The sign of $\mu(h)-u^*(h)$  with respect to the parameter $h$ remains to be found.\\
On a first hand, the explicit expression of $f_h'$ shows that $f'_{h^*}$  is decreasing and that $f'_{h^*}(\mu(h^*))=0$. Moreover, the definition \eqref{defhstar} of   $h^*$ yields $u^*(h^*)=u_{crit}$ and $f_{h^*}'(u_{crit})=g_{h^*}'(u_{crit})>0$. Hence, $\mu(h^*)<u^*(h)$. It is also clear that $\frac12=\lim\limits_{h\to +\infty} \mu(h)<\lim\limits_{h\to +\infty} u^*(h)=1$. By continuity, we infer that the equation $\mu(h)=u^*(h)$ has at least one solution in $(h^*,+\infty)$. \\
On another hand, if $\mu(h)=u^*(h)$ then $H(h,\mu(h))=0$, which may be rewritten as
$$(Eh+1)^3=4E((Eh)^2+Eh(E\alpha+1)-E\alpha).$$
A simple analysis shows that this equation has exactly one negative solution, one solution in $\left(0,\frac{1}{E}\right)$ and one solution in $\left(\frac{1}{E},+\infty\right)$. Since $h^*>\frac{1}{E}$
 (see  footnote \ref{minhstar}),  this implies that $\mu(h)=u^*(h)$ has exactly one solution in $(h^*,+\infty)$. We note this unique solution $h^{**}=h^{**}(E,\alpha)$.\\
Finally, it is clear from the above arguments  that $\mu(h)<u^*(h)$ if $h\in(h^*,h^{**})$ and  that $\mu(h)>u^*(h)$ if $h>h^{**}$. This ends the proof of the Theorem.\qed

\paragraph{Proofs of properties \ref{prophet}}
Similarly to  the proof of the theorem \ref{thODE}, and  to highlight the role of the parameters $E$ and $\alpha$, let us define
$$H(E,\alpha,h,u)=\frac{1}{E}(1-u)(1+Ehu)-\left(1+\alpha \frac{Eu}{1+Ehu}\right).$$
From the proof of the theorem \ref{thODE}, we know that the quantity $h^*=h^*(E,\alpha)$ and the corresponding $u_{crit}=u_{crit}(E,\alpha)$ are characterized by the two  equations
\begin{subequations}\label{hucrit}
\begin{gather}
\phantom{\partial_u}H(E,\alpha,h^*,u_{crit})=0\label{hucrit:1}\\
\partial_u H(E,\alpha,h^*,u_{crit})=0.\label{hucrit:2}
\end{gather}
\end{subequations}
The implicit function theorem immediatly shows that the maps $(E,\alpha)\mapsto (h^*,u_{crit})$ belongs to $C^1\big((1,+\infty)\times [0,+\infty); \R_+^2\big)$.
\begin{itemize}
\item {\bf Proofs of the growth of $E\mapsto h^*(E,\alpha)$ and of $\alpha\mapsto h^*(E,\alpha)$.}\\
Let $\alpha\geq 0$ be fixed.
Differentiate \eqref{hucrit:1} with respect to $E$ and use \eqref{hucrit:2} yields 
$$\partial_E H(E,\alpha,h^*,u_{crit})+\partial_E h^*(E,\alpha)\cdot \partial_h H(E,\alpha,h^*,u_{crit})=0.$$
We already know that $\partial_h H<0$ and an explicit computation gives
$$\partial_E H(E,\alpha,h^*,u_{crit})=\frac{1}{E (1+Eh^* u_{crit})}>0.$$
It follows that $\partial_E h^*(E,\alpha)>0$.
Similar arguments show that $\partial_\alpha h^*(E,\alpha)>0$.

\item {\bf Computation of ${\displaystyle \lim_{E\to 1}h^*(E,\alpha)}$.}\\
Let $\alpha\geq 0$ be fixed. Since $h^*(\cdot,\alpha)$ is increasing and positive on $(1,+\infty)$, there exists a nonnegative scalar $h^*(\alpha)$ such that
 $h^*(E,\alpha)\to h^*(\alpha)$ as $E\to 1$. To compute this limit, denote $P(E,\alpha,h,u)=(1+Ehu) H(E,\alpha,h,u)$. From \eqref{hucrit} and the definition of $h^*$, one see  that $h^*$ is the minimal value of $h$ such that
\begin{equation}\label{Phstar}\exists u\in [0,1],\; P(E,\alpha,h,u)=\partial_u P(E,\alpha,h,u)=0.\end{equation}
In other words, $h^*$ is the minimal value of $h$ such that $P(E,\alpha,h,\cdot)$ admits a multiple root in $[0,1]$.\\
 By passing to the limit $E\to 1$ in \eqref{Phstar}, we see that that $h^*(\alpha)$  is the minimal value of $h$ such that  
\begin{equation}\label{Phstar1}\exists u\in [0,1],\; P(1,h,\alpha,u)=\partial_u P(1,h,\alpha,u)=0.\end{equation}
Explicit computations give
 $$P(1,\alpha,h,u)=u\left( h^2 u^2+h(2-h)u+\alpha+1-h\right).$$
The multiplicity of the roots of $P(1,\alpha,h,\cdot)$ need now to be discussed.

$\bullet$ If $h<2\sqrt{\alpha}$ then $0$ is the only root of $P(1,\alpha,h,\cdot)$ and the multiplicity  is one.\\
$\bullet$ If $h\geq 2\sqrt{\alpha}$ then $P(1,\alpha,h,\cdot)$ has two other real roots and may be explicitly written as
$$P(1,\alpha,h,u)=h^2 u(u-u_-)(u-u_+)$$ where
$$u_{\pm}=\frac{1}{2h}\left(h-2\pm\sqrt{(2-h)^2-4(\alpha+1-h)}\right).$$
\begin{itemize}
\item If $h=2\sqrt{\alpha}$  then   $u_{-}=u_{+}=1-\frac{1}{\sqrt{\alpha}}$ and this multiple root belongs to $[0,1)$ if and only if $\alpha\geq 1$.\\
\item If $2\sqrt{\alpha}<h<\alpha+1$ then the three roots $0$, $u_-$ and $u_+$ are distinct and there is no multiple root.\\
\item If $h=\alpha+1$ then either $0=u_-$ or $0=u_+$, depending on the sign of $\alpha-1$. In both cases $0$ is a multiple root.\\
\item Finally, if $h>\alpha+1$ then $u_-<0<u_+$ and all the roots of $P(1,\alpha,h,\cdot)$ have multiplicity one.\\ 
\end{itemize}
The above discussion shows, using the characterization of $h^*(\alpha)$, that 
\begin{equation}\label{limithetEto1}
h^*(\alpha)=\left\{\begin{array}{l} 1+\alpha\text{ if } \alpha\leq 1\\ 2\sqrt{\alpha}\; \text{ if } \alpha> 1.\end{array}\right.
\end{equation}
Note that $\alpha\mapsto h^*(\alpha)$ belongs to $C^1([0,+\infty),\R^+)$.  \qed
\item {\bf Computation of ${\displaystyle \lim_{E\to +\infty}h^*(E,\alpha)}$.}\\
Since $h^*(\cdot,\alpha)$ is positive and increasing, one has 
$$\lim_{E\to +\infty}h^*(E,\alpha) = \frac{1}{\ell_\alpha} $$   for some nonnegative number
$\ell_\alpha$ (wherein we have set $\frac{1}{0}=+\infty$).\\

However, since $u_{crit}(E,\alpha)$ is bounded, up to a subsequence again denoted by $E$, one has $u_{crit}(E,\alpha)\to \mu$ 
for some positive number $\mu$ (eventually depending on $\alpha$). By passing to the limit in \eqref{hucrit:1}, we deduce $\mu>0$; for otherwise one obtains $0=H(E,\alpha,h^*,u_{crit})\to +\infty$. It follows that taking $E\to+\infty$ in \eqref{hucrit:2}, one obtains $\mu=1/2$. Thus, by passing to the limit in \eqref{hucrit:1}, one gets
 $$\frac{1}{\ell_\alpha^2}=\frac{2}{\ell_\alpha}+2\alpha$$ and finally ($\ell_\alpha$ being nonnegative)
$$\frac{1}{\ell_\alpha}=2+2\sqrt{1+\alpha};$$ \qed
\item {\bf Computation of ${\displaystyle h^*(E,0)}$.}
If $\alpha=0$  one gets 
\begin{equation}\label{alpha0}
P(E,h,0,u)=\frac{1}{E}(1+Ehu)\left(-Ehu^2+u(Eh-1)+1-E\right).
\end{equation}
Standard computations show that $P(E,h,0,\cdot)$ admits two nonnegative roots  if and only if $h>\frac{1}{E}(2E-1+\sqrt{E(E-1)}:=h_1(E)$ and one double nonnegative root if $h=h_1(E)$.
This shows  that $h^*(E,0)=h_1(E)$. \qed
\item {\bf Computation of  ${\displaystyle \lim_{\alpha\to +\infty}h^*(E,\alpha)}$.}
Since $E\mapsto h^*(E,\alpha)$ is increasing one gets for any $E\geq 1$, $h^*(E,\alpha)\geq h^*(\alpha)$. 
The explicit expression \eqref{limithetEto1} of $h^*(\alpha)$ implies $ \lim\limits_{\alpha\to +\infty}h^*(E,\alpha)=+\infty.$\qed
\end{itemize}
\subsection{Proof of the Theorem \ref{thaspcontrol}.}\label{proofaspcontrol}
Let $E>1$ be fixed and $(u(t,x),v(t,x))$ be a solution of \eqref{sys2} with initial condition $(u_0(x),v_0(x))$ verifying \eqref{CI}.\\
\paragraph{Step 1.}
It is clear that $u(t,x)\geq 0$ for any $t>0$ which implies
$$\partial_t v(t,x)-d\Delta_x v(t,x)\geq r v(t,x)(1-v(t,x)),\quad t>0,\quad x\in \R.$$
Thus, using the comparison principle and the hypothesis that $v(0,x)\geq 1$, we deduce 
$v(t,x)\geq 1$ for any $t\geq 0$ and $x\in\R$. It follows that
$$\partial_t u(t,x)-\Delta_x u(t,x)\leq u(t,x)(1-u(t,x))-\frac{Eu(t,x)}{1+Ehu(t,x)},\quad t>0,\quad x\in \R.$$
By the comparison principle, we infer that any  solution $\overline{u}$  of 
\begin{equation}\label{scalarv1}
\partial_t \overline{u}(t,x)-\Delta_x \overline{u}(t,x)= \overline{u}(t,x)(1-\overline{u}(t,x))-\frac{E\overline{u}(t,x)}{1+Eh\overline{u}(t,x)},
\; t>0,\quad x\in \R\end{equation}
such that $\overline{u}(0,x)\geq u_0(x)$ for any $x\in \R$ satisfies
$$\forall t>0,\quad \forall x\in \R,\quad u(t,x)\leq \overline{u}(t,x).$$
In particular, let $\phi(t)$ be the solution of the ordinary differential equation
\begin{equation}\label{scalarv1ode}
\frac{d}{dt} \phi(t)= \phi(t)(1-\phi(t))-\frac{E\phi(t)}{1+Eh\phi(t)},\quad t>0
\end{equation}
together with the initial condition $\phi(0)=1$.\\
$\phi$ is a homogeneous solution of \eqref{scalarv1} and from $u(0,x)\leq 1=\phi(0)$  we deduce
$$\forall t>0,\quad \forall x\in \R,\quad 0\leq u(t,x)\leq \phi(t).$$
\paragraph{Step 2.}
Behavior of $\phi(t)$ as $t\to+\infty$. \\
It is clear that $0$ is always a steady state of \eqref{scalarv1ode} and is asymptotically stable for $E>1$.\\ 
Now, let $\phi_0> 0$ be a positive steady state of \eqref{scalarv1ode}. $\phi_0$ is a root of the polynomial $P(E,h,0,\cdot)$ which is studied in the proof of the property \ref{prophet} (see \eqref{alpha0}).
Hence, if $h\geq h_1(E)$ then \eqref{scalarv1ode} has two positive steady states that we denote as $u^-(E,h)\leq u^+(E,h)\leq1$ explicitly given by
\begin{equation}\label{upm}
u^{\pm}(E,h)=\frac{1}{2}\left(1-\frac{1}{Eh}\pm\sqrt{\left(1-\frac{1}{Eh}\right)^2-4\frac{E-1}{Eh}}\right).
\end{equation}
A linear analysis shows  that for $h>h_1(E)$,  $u^-(E,h)$ is unstable and $u^+(E,h)$ is (asymptotically) stable.\\
Finally, classical arguments show that $\phi(t)\to u^+(E,h)$ if $h>h_1(E)$ and $\phi(t)\to 0$ if $h<h_1(E)$.  This ends the proof of the theorem. \qed\\

\subsection{Proof of the Theorem \ref{thcontrol}.}\label{proofcontrol}
\noindent Let $E>1$ be fixed and $(u(t,x),v(t,x))$ be a solution to \eqref{sys2} with initial condition $(u_0(x),v_0(x))$  verifying \eqref{CI}.\\
By the  argument of the step 1 of the proof  of \ref{proofaspcontrol}, one already knows that: 
\begin{equation}\label{abovecontrol}
u(t,x)\leq \overline{u}(t,x)
\end{equation}
where $\overline{u}(t,x)$ is a solution of \eqref{scalarv1} with $\overline{u}(0,x)=u_0(x)$. Moreover, from the step 2 of the proof of \ref{proofaspcontrol}, one knows
that  if $h>h_1(E)$ then the equation \eqref{scalarv1} is {\it bistable} since \eqref{scalarv1} admits two stable nonnegative steady states: $u=0$ and $\overline{\mu}=u^+(E,h)<1$.
We prove here that, in that case, there exists  a traveling wave connecting  $\overline{\mu}$ to $0$ at a negative speed if and only 
if $h<h^-(E)$ 
for some (implicit) number $h^-(E)>h_1(E)$.
This  implies that for any $x\in\R$,  $\overline{u}(t,x)\to 0$ if $h<h^-(E)$ and $\overline{u}(t,x)\to \overline{\mu}$ if $h>h^{-}(E)$,  which proves  the theorem.\\

Let $h>h_1(E)$ be fixed. It can be proven (see \cite{fife}, \cite{volperts} and the reference therein) that  there exists a unique speed $c=c(E,h)$ such that  \eqref{scalarv1} admits a traveling solution of speed $c$ which connects   
$\overline{\mu}$ to $0$. More precisely, there exists a profile  $U(\xi)$ verifying $U(-\infty)=\overline{\mu}$, $U(+\infty)=0$ and $U'(\pm \infty)=0$, such that $U(x-ct)=u(x,t)$ is a solution of \eqref{scalarv1}. 
Moreover (see \cite{fife}) this traveling wave describes the 
asymptotic behavior of all solutions provided the initial condition \eqref{CI} are verified.
The theorem is  proven by showing that there exists $h^{-}(E)$ such that $c(E,h)<0$ if $h<h^{-}(E)$ and $c(E,h)>0$ if  $h>h^{-}(E)$.
This result is a direct consequence of  the two following lemma.
The first lemma gives a characterization of the sign of $c$ by an explicit function (see figure \ref{fig:W}).
\begin{lemma}\label{lemmaspeed}
Define 
$$W(E,h,u))=\int_{0}^{u} \left(s(1-s)-\frac{Es}{1+Ehs}\right)ds.$$
Then $sign(c(E,h))=sign(W(E,h,\overline{\mu}))$ where $\overline{\mu}:=u^+(E,h)$.
\end{lemma}
\begin{proof}[Proof]
Denote $c=c(E,h)$ and let  $\xi=x-ct$ and $U(\xi)=u(x,t)$ with $U(-\infty)=\overline{\mu}$, $U(+\infty)=0$ and $U'(\pm \infty)=0$. We have
$$-cU'=U''+U(1-U)-\frac{EU}{1+EhU}=U''+\frac{\partial W}{\partial U} (E,h,U).$$
Multiplying by $U'$ and integrating over $\mathbb{R}$ one gets
\begin{align*}
-c\int_{-\infty}^{+\infty} (U')^2 dZ&=\int_{-\infty}^{+\infty}\left(U(1-U)-\frac{EU}{1+EhU}\right)U'dZ\\
&=\int_{\overline{\mu}}^0\frac{\partial W}{\partial U} (E,h,U)dU=-W(E,h,u^+)\\
\end{align*}
and $sign(c)=sign(W(E,h,u^+(E,h))$ follows.\qed
\end{proof}

\begin{figure}[htb]

\centering
\begin{tikzpicture}[x=3cm,y=12cm,>=stealth]
\def\a{0.55}
\draw [->] (-0.4,0) -- (1.4,0) node [below]{$u$};
\draw [->] (0,-0.1) -- (0,0.1);
\draw plot [domain=-0.2:1.3,smooth] (\x, {(\x)^4-(4/3)*(1+\a)*(\x)^3+2*(\a)*(\x)^2}); 

\draw (-0.05,0) node [below] {$0$};
\draw (\a,0) node [below] {$u^-$};
\draw (1,0) node [below] {$u^+$};
\draw (0.5,0.15) node [above] {\small $-W(u^+)>0$ and then $c<0$};
\draw [dotted](\a,0) --(\a,{(\a)^4-(4/3)*(1+\a)*(\a)^3+2*(\a)*(\a)^2});
\draw [dotted](1,0) --(1,{(1)^4-(4/3)*(1+\a)*(1)^3+2*(\a)*(1)^2});

\end{tikzpicture}
$\qquad \qquad$
\begin{tikzpicture}[x=3cm,y=12cm,>=stealth]
\def\a{0.4}
\draw [->] (-0.4,0) -- (1.4,0) node [below] {$u$};
\draw [->] (0,-0.1) -- (0,0.1);
\draw plot [domain=-0.2:1.3,smooth] (\x, {(\x)^4-(4/3)*(1+\a)*(\x)^3+2*(\a)*(\x)^2}); 
\draw (-0.05,0) node [below] {$0$};
\draw (\a,0) node [below] {$u^-$};
\draw (1,0) node [above] {$u^+$};
\draw (0.5,0.15) node [above] {\small $-W(u^+)<0$ and then $c>0$};
\draw [dotted](\a,0) --(\a,{(\a)^4-(4/3)*(1+\a)*(\a)^3+2*(\a)*(\a)^2});
\draw [dotted](1,0) --(1,{(1)^4-(4/3)*(1+\a)*(1)^3+2*(\a)*(1)^2});
\end{tikzpicture}

\caption{\label{fig:W}Graph of $u\mapsto -W(E,h,u)$. The steady states $0$, $u^-$ and $u^+$ correspond to critical points of the potential $-W$. 
A stable steady state is a local minimum of $W$  and an unstable steady state  is a local maximum.
One sees that $0$ and $u^+$ are two stable steady states. The sign of $c$ characterizes which of them is 
the final global attractor: if $c<0$, then $0$ is the global attractor.
If $c>0$, then $u^+$ is the global attractor. 
}
\end{figure}
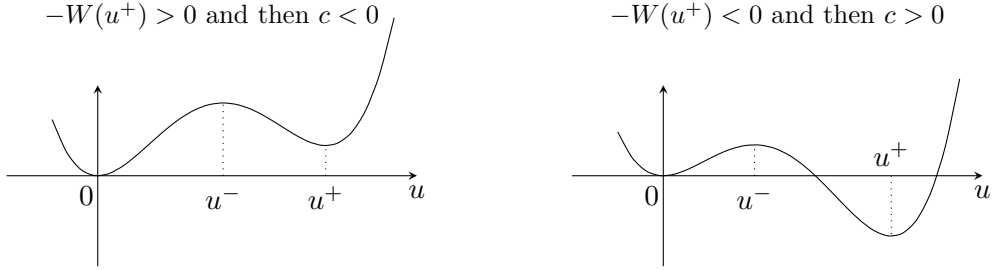
The second lemma gives the sign of $W$ with respect to $h$.
\begin{lemma}\label{finallemma}
For any $E>1$, there exists (a unique) $h^-(E)>h_1(E)$ such that 
$$sign(W(E,h,u^+(E,h)))=sign\left(h-h^-(E)\right).$$
\end{lemma}
Lemma \ref{lemmaspeed} and \ref{finallemma} show  that
\begin{itemize}
\item There is invasion of prey for \eqref{scalarv1}  ($c(E,h)>0$) if $h\in (h^-(E),+\infty)$.
In that case, for any $x\in\R$, $\overline{u}(t,x)\to u^+(E,h)$ as $t\to+\infty$.
\item There is extinction of prey for  \eqref{scalarv1} ($c(E,h)<0$) if $h\in [h_1(E),h^-(E))$.
In that case, for any $x\in\R$, $\overline{u}(t,x)\to 0$ as $t\to+\infty$.
\end{itemize}
In particular, we infer from the inequality  \eqref{abovecontrol}, that if $h\in [h_1(E),h^-(E))$, then  for any $x\in\R$, $u(t,x)\to 0$ as $t\to+\infty$. This ends the proof of the theorem.\qed\\

\noindent It remains to prove the lemma  \ref{finallemma}.
\paragraph{Proof of lemma \ref{finallemma}.}
Define $\mathcal{W}(E,h)=W(E,h,u^+(E,h))$. From the  lemma \ref{lemmaspeed}, we know that  $sign(c)=sign(\mathcal{W}(E,h))$. We show here that there exists 
$h^-(E)>h_1(E)$ such that $sign(\mathcal{W}(E,h))=sign(h-h^-(E))$.\\
\paragraph{Step 1.} Differentiate $\mathcal{W}$ with respect to $h$ gives
$$\partial_h \mathcal{W} (E,h)=\partial_h W (E,h,u^+(E,h))+\partial_u W(E,h,u^+(E,h)) \partial_h u^+(E,h).$$
By the very definition of $W$ and $u^+$, one has
$\partial_u W(E,h,u^+(E,h))=0$, which yields 
$$\partial_h \mathcal{W}(E,h)=\partial_h W (E,h,u^+(E,h)).$$
Explicit calculations show that
\begin{equation}\label{Wexplicit}
W(E,h,u)=\frac{u^2}{2}-\frac{u^3}{3}- \frac{u}{h}+\frac{1}{E h^2} \ln(1+Ehu).\end{equation}
Differentiate this expression with respect to $h$ and denoting $z=Ehu^+(E,h)$ provides
$$\partial_h W(E,h,u^+(E,h))=\frac{1}{Eh^3}\left(z-2\ln(1+z)+\frac{z}{1+z}\right).$$
A standard analysis shows that the map $z\mapsto z-2\ln(1+z)+\frac{z}{1+z}$ takes positive values for $z>0$.
It follows that  the map $h\mapsto \mathcal{W}(E,h)$ is increasing. 
\paragraph{Step 2.} Recalling  that  $u^+$ is the largest roots of \eqref{alpha0}, one verifies that $u^+(E,h)\to 1$ as $h\to +\infty$ and  $\mathcal{W}(E,+\infty)=1/2-1/3=1/6>0$.
\paragraph{Step 3.}This step consists in proving that for any $E>1$, $\mathcal{W}(E,h_1(E))=W(E,h_1(E),u^+(E,h_1(E))):=g(E)$ is negative.\\

\noindent From the explicit expression \eqref{upm} of $u^+(E,h)$ and the definition of  $h_1(E)$, we get
$$u^+(E,h_1(E))=\frac{1}{2}\left(1-\frac{1}{Eh_1(E)}\right)=1-E+\sqrt{E(E-1)}.$$
On a first hand, we have
$$g'(E)=\partial_E W(E,h_1(E),u^+(E,h_1(E))) +\partial_h W(E,h_1(E),u^+(E,h_1(E)))\cdot h_1'(E).$$
Straightforward calculations give
%
$$g'(E)=\left(\frac{1}{Eh_1(E)}\right)^2\left[2-\ln(1+z)\left(1+\frac{2}{z}\right)\right].$$
wherein we have set
$$z=Eh_1(E)u^+(E,h_1(E))=\sqrt{E-1}\left(\sqrt{E}+\sqrt{E-1}\right)=\frac12 \left(Eh_1(E)-1\right).$$
A standard analysis shows that $2\frac{z}{z+2}<\ln(1+z)$ for any $z>0$. Thus  $g'(E)<0$ for any $E>1$. \\
   On another hand, we have $h_1(1)=1$ and $u^+(1,1)=0$, so that  $g(1)=0$. It follows that  $g(E)<0$ for any  $E>1$.\\
\paragraph{Conclusion}
For any $E> 1$, the map $h\mapsto \mathcal{W}(E,h)$ is  increasing and verifies  $\mathcal{W}(E,h_1(E))<0$ and $\lim\limits_{h\to+\infty}\mathcal{W}(E,h)=1/6$.
By continuity, for any $E>1$, there exists a unique $h^-=h^-(E)\in (h_1(E),+\infty) $ such that $sign(W(E,h,u^+))=sign(h-h^-(E))$. This ends the proof of  lemma \ref{finallemma}.\qed

\paragraph{Proofs of properties \ref{prophm}.}
\noindent Let $E>1$ be fixed.
\begin{itemize}
 \item {\bf Growth of $h^-(\cdot)$.}\\
 Recall that $h^-(E)$ is characterized by
\begin{equation}\label{caracthmoins}
W(E,h^-(E),u^+(E,h^-(E)))=0
 \end{equation}
 where the expressions of $W$ and $u^+$ are respectively given in \eqref{Wexplicit} and \eqref{upm}.
By definition of $W$, one has
$$\partial_u W\left(E,h^-(E),u^+(E,h^-(E))\right)=0.$$
Hence
$$\partial_E W\left(E,h^-(E),u^+(E,h^-(E))\right)+\frac{d h^-}{dE} (E)\cdot \partial_h W\left(E,h^-(E),u^+(E,h^-(E))\right)=0.$$
The explicit computation of $\partial_E W$ and $\partial_h W$ are done in proof \ref{proofcontrol} and one gets
$$\partial_E W\left(E,h^-(E),u^+(E,h^-(E))\right)<0 \text{ and } \partial_h W\left(E,h^-(E),u^+(E,h^-(E))\right)>0$$
so that 
$$\frac{d h^-}{dE} (E)=-\frac{\partial_E W\left(E,h^-(E),u^+(E,h^-(E))\right)}{ \partial_h W\left(E,h^-(E),u^+(E,h^-(E))\right)}>0$$
as needed.\qed

\item {\bf Limit of $h^-(E)$ as $E\to 1$.}\\
\noindent Let $E>1$. One knows that $h^{-}(\cdot)$ is increasing on $(1,+\infty)$ and minored by $h_1(E)>0$. Hence $h^-(E)$ admits a limit $\ell^-$ as $E\to 1$, and from $h_1(1)=1$, one obtains
$$\ell^-\geq 1.$$
The explicit expressions \eqref{upm} of $u^\pm$ yield
$$u^{-}(1,\ell^-)=0\text{ and }u^+(1,\ell^-)=\frac{1-\ell^-}{\ell^-}.$$
Assume by contradiction that $u^+(1,\ell^-)>0$. Then, 
for any $h\geq 1$ and $0< u<u^+(h,1)$, we have
$$\partial_uW(1,h,u)=u(1-u)-\frac{1}{1+hu}>0 \text{ and }   W(1,h,0)=0.$$  
This implies $W(1,\ell^-,u^+(1,\ell^-))<0$, which contradicts \eqref{caracthmoins}. \\
It follows that $u^+(1,\ell^-)=0$. Thus $\ell^-=1$, which reads
$\lim\limits_{E\to 1} h^-(E)=1.$\qed

\item {\bf Limit of $h^-(E)$ as $E\to +\infty$.}\\
One knowns that $h^-(E)$ is increasing so that there is some $\ell\geq0$ such that 
${\displaystyle\lim_{E\to +\infty}h^-(E)= \frac{1}{\ell}}$ (wherein we have set $\frac{1}{\infty}=0$). 
By taking  the limit $E\to +\infty$ in \eqref{upm}, we obtain
$$\lim_{E\to+\infty} u^+(E,h^-(E))=u_\ell:=\frac{1}{2}(1+\sqrt{1-4\ell})>0.$$
and then, by  taking  the limit in \eqref{caracthmoins},
$$\frac{u_\ell^2}{2}-\frac{u_\ell^3}{3}-\frac{u_\ell}{\ell}=0. $$
Thus, $\ell=\frac{3}{16}$ which reads, $\lim\limits_{E\to+\infty} h^-(E)=\frac{16}{3}.$
\end{itemize}
\subsection{Invasion conditions, proof of Theorem \ref{thnocontrol}}\label{proofnocontrol}
As in the proof \ref{proofcontrol}, we start here by showing that 
 $0\leq \underline{u}(t,x)\leq u(t,x)$ for some function $\underline{u}$ verifying a {\it scalar}
reaction-diffusion equation \eqref{scalarvbar}  depending on $E$, $h$ and $\alpha$. 
Next, we show that there exists $h^+(E,\alpha)$ such that $\lim_{t\to+\infty}\underline{u}(t,x)=\underline{\mu}>0$ when $h>h^+(E,\alpha)$.
This step is done using the proof of Theorem \ref{thcontrol} together with an appropriate change of variables.\\

\paragraph{Step 1.}
\noindent Let $E>1$ be fixed.
In general, using $0\leq u\leq 1$,  one has $1\leq v\leq \overline{v}$ with 
$$\overline{v}:=\overline{v}(E,h,\alpha)=1+\alpha\frac{E}{1+Eh}.$$ 
\\
From the estimate $v\leq \overline{v},$ we get $\underline{u}(t,x)\leq u(t,x)$ where $\underline{u}$ 
verifies 
\begin{equation}\label{scalarvbar}
\partial_t u=\Delta_x u +u(1-u)-\frac{E\overline{v}u}{1+Ehu},\quad t>0,\quad x\in \R.
\end{equation}
Denoting $\widetilde{E}=E\overline{v}$ and $\widetilde{h}=\frac{h}{\overline{v}}$, this equation reads simply
\begin{equation}\label{scalarv1tilda}
\partial_t u=\Delta_x u +u(1-u)-\frac{\widetilde{E}u}{1+\widetilde{E}\widetilde{h}u},\quad t>0,\quad x\in \R. 
\end{equation}
Note that, removing the $\widetilde{}$, this equation is nothing but \eqref{scalarv1}.
As a consequence, the proof of the theorem \ref{thcontrol} implies the following result on \eqref{scalarv1tilda}.
\begin{lemma}
Let $\widetilde{u}(t,x)$ be a solution of \eqref{scalarv1tilda} verifying the initial condition \eqref{CI}. 
There exists $h^-(\widetilde{E})>h_1(\widetilde{E})$ such that 
\begin{itemize}
\item if 
$\widetilde{h}<h^-(\widetilde{E})$ then for any $x\in \R$, $\lim\limits_{t\to+\infty}\widetilde{u}(t,x)= 0$ (extinction),
\item if $\widetilde{h}>h^-(\widetilde{E})$ then for any $x\in\R$, $\lim\limits_{t\to+\infty}\widetilde{u}(t,x)= \overline{\mu}(\widetilde{E},\widetilde{h})>0$ (invasion).
\end{itemize}
\end{lemma}
\paragraph{Step 2.}
Recalling that $\widetilde{E}=E \overline{v}(E,h,\alpha)$  and $\widetilde{h}=\frac{h}{ \overline{v}(E,h,\alpha)}$ 
do depend on $h$, we set $\underline{\mu}=\underline{\mu}(E,h)=\overline{\mu}(\widetilde{E},\widetilde{h})$ and
the previous lemma gives 
 the following implicit condition on $h$ for invasion to  occur.\\
For any $x\in\R$, $\lim\limits_{t\to+\infty}u(t,x)= \underline{\mu}$, provided:
\begin{equation}\label{implicitcondition}
h> \overline{v}(E,h,\alpha)h^-\left(E  \overline{v}(E,h,\alpha)\right).
\end{equation}
The following lemma gives an equivalent condition for this implicit condition to occur. 
This ends the proof of  theorem \ref{thcontrol}.
\begin{lemma}\label{lemmadefh+} For any $E>1$ and $\alpha\geq 0$, there exists (a unique) $h^+(E,\alpha)\geq h^-(E)$ such that:\\
 $$\text{\eqref{implicitcondition} holds true if and only if $h>h^+(E,\alpha)$.}$$
\end{lemma}
\noindent {\bf Proof of lemma \ref{lemmadefh+} :} 
Let $E>1$ and $\alpha\geq 0$ be fixed and define the function 
\begin{equation}\label{fEalpha}
\mathcal{F}_{E,\alpha}(h)=h- \overline{v}(E,h,\alpha)h^-\left(E  \overline{v}(E,h,\alpha)\right).
\end{equation}
By the construction of $h^-$ via the implicit function theorem, one knows that $\mathcal{F}_{E,\alpha}(\cdot)$ is a $C^1$ function.
Moreover, since $h^-$ and $\overline{v}$ are  bounded, we have
$\mathcal{F}_{E,\alpha}(+\infty)=+\infty$ and 
$$\mathcal{F}_{E,\alpha}(0)=-\overline{v}(E,0,\alpha)h^-(E\overline{v}(E,0,\alpha))=-(1+\alpha E)h^-(E(1+\alpha E))<0.$$
Therefore, it suffices to show that $\mathcal{F}_{E,\alpha}$ is increasing.\\
One has
$$\mathcal{F}_{E,\alpha}'(h)=1-\partial_h\overline{v}(E,h,\alpha)\left(h^-\left(E  \overline{v}\right)+\overline{v}\frac{dh^-}{dE}(E\overline{v})\right).$$
From the expression of $\overline{v}$, we infer $\partial_h\overline{v}<0$ and from the proof of the properties \ref{prophm}, we know that 
$\frac{dh^-}{dE}>0$. It follows that $\mathcal{F}_{E,\alpha}'(h)>0$. This ends the proof of the lemma.
\qed

\paragraph{Proofs of properties \ref{prophp}.}

\begin{itemize}
 \item {\bf Proof of the growth of   $h^+(E,\cdot)$ and $h^+(\cdot,\alpha)$.}\\
$h^+(E,\alpha)$ is characterized by an equality in \eqref{implicitcondition}, that is, 
\begin{equation}\label{defh+}
\mathcal{F}_{E,\alpha}(h^+(E,\alpha))=0
\end{equation}
where $\mathcal{F}_{E,\alpha}$ is defined in \eqref{fEalpha}.
Differentiating \eqref{defh+} with respect to $E$ gives 
$$\partial_E h^+(E,\alpha)\cdot \mathcal{F}'_{E,\alpha}(h^+(E,\alpha) )+\partial_E \mathcal{F}_{E,\alpha} (h^+(E,\alpha) )=0$$
and then
$$\partial_E h^+(E,\alpha)\cdot  \mathcal{F}'_{E,\alpha}(h^+(E,\alpha) )= \left[(E\overline{v}+h^-)\partial_E \overline{v}+\overline{v}^2\right]
\cdot \frac{dh^-}{dE}\left(E  \overline{v}\right).$$
wherein we have set
$$h^-=h^-\left(E  \overline{v}\right),\quad\overline{v}=\overline{v}(E,h^+(E,\alpha),\alpha)\text{ and }
\partial_E\overline{v}=\partial_E\overline{v}(E,h^+(E,\alpha),\alpha).$$
One already knows that $\mathcal{F}'_{E,\alpha}(h^+(E,\alpha) )>0$ and that $\frac{dh^-}{dE}>0$. 
A direct computation shows that $\partial_E\overline{v}>0$, which leads to $\partial_E h^+(E,\alpha)>0$ as needed.\\

Similarly, differentiating \eqref{defh+} with respect to $\alpha$ gives, with  obvious notations,

$$\partial_\alpha h^+(E,\alpha)\cdot \mathcal{F}'_{E,\alpha}(h^+(E,\alpha) )=\partial_\alpha\overline{v}\cdot \left(h^-+E\frac{dh^-}{dE}\right),$$
and since $\partial_\alpha\overline{v}>0$, one obtains $\partial_\alpha h^+(E,\alpha)>0$.\\

\item {\bf Limits of $h^+(E,\alpha)$ as $\alpha\to0$ and $\alpha \to +\infty$.}\\ 
From $\overline{v}(E,h,0)=1$ we deduce $h^+(E,0)=h^-(E)$.
Now, it is clear from the construction of $h^+$, that $h^+(E,\alpha)\geq h^*(E,\alpha)$. From the properties \ref{prophet}, we obtain
$h^+(E,\alpha) \to +\infty$ as $\alpha\to +\infty.$
\qed

\item {\bf Limit of $h^+(E,\alpha)$ as $E\to+\infty$.}\\
First, recall that $h^+(E,\alpha)$ is characterized by \eqref{defh+}, which reads
\begin{equation}\label{caracthplus}
h^+(E,\alpha)=\left(1+\alpha\frac{E}{1+Eh^+(E,\alpha)}\right)\cdot h^-\left(E+\alpha\frac{E^2}{1+Eh^+(E,\alpha)}\right).
 \end{equation}
Since $h^+(\cdot,\alpha)$ is increasing, there exists $\ell_\alpha\geq 0$ such that 
$\lim_{E\to+\infty} h^+(E,\alpha)=\frac{1}{\ell_\alpha}.$
Taking  the limit $E\to+\infty$ in \eqref{caracthplus} and using  $h^-(E)\to \frac{16}{3}$, one obtains
$\frac{1}{\ell_\alpha}=\frac{16}{3} \cdot (1+\alpha \ell_\alpha)$. The resolution of this equation ends the proof.
\qed
\end{itemize}

\noindent {\bf Acknowledgements}
The authors would like to thank the two anonymous reviewers for their valuable comments and suggestions to improve the quality of this manuscript.

\end{document}